\documentclass[10pt,a4paper,reqno]{amsart}

\usepackage[utf8]{inputenc}
\usepackage[T1]{fontenc}
\usepackage{lmodern}

\usepackage{mathtools}
\usepackage{amssymb}
\usepackage{bbm}
\usepackage[dvipsnames]{xcolor}
\usepackage{enumitem}
\usepackage{verbatim}

\usepackage{amsthm}
\usepackage{hyperref}
\hypersetup{
    colorlinks=true,
    citecolor=blue,
    urlcolor=blue
}
\usepackage[nameinlink,capitalize]{cleveref} 

\numberwithin{equation}{section}
\theoremstyle{plain}
\newtheorem{ttheorem}{Theorem}
\newtheorem{theorem}{Theorem}[section]
\newtheorem{corollary}[theorem]{Corollary}
\newtheorem{lemma}[theorem]{Lemma}
\newtheorem{proposition}[theorem]{Proposition}
\newtheorem{hypothesis}{Hypothesis}
\theoremstyle{remark}
\newtheorem{remark}[theorem]{Remark}
\newtheorem{example}[theorem]{Example}
\newtheorem{definition}[theorem]{Definition}

\newcommand{\R}{{\mathbb R}}
\newcommand{\N}{{\mathbb N}}

\renewcommand{\P}{{\mathbb P}}

\newcommand{\calA}{\mathcal A}

\newcommand{\calF}{\mathcal F}

\newcommand{\calI}{\mathcal I}
\newcommand{\calL}{\mathcal L}

\newcommand{\calN}{\mathcal N}
\newcommand{\calO}{\mathcal O}

\newcommand{\eps}{\varepsilon}

\newcommand{\ind}{\mathbbm{1}}
\newcommand{\rd}{{\, \rm d}}

\newcommand{\Euc}{\mathrm{Euc}}
\newcommand{\Cheb}{\mathrm{Cheb}}

\DeclareMathOperator{\card}{card}

\DeclarePairedDelimiter{\bra}{(}{)}
\DeclarePairedDelimiter{\sqr}{[}{]}
\DeclarePairedDelimiter{\cur}{\{}{\}}
\DeclarePairedDelimiter{\norm}{\lVert}{\rVert}
\DeclarePairedDelimiter{\abs}{\lvert}{\rvert}
\DeclarePairedDelimiter{\ceil}{\lceil}{\rceil}
\newcommand{\EE}[2][\big]{\mathbb{E} \sqr[#1]{#2}}
\newcommand{\PP}[2][\big]{\mathbb{P} \bra[#1]{#2}}

\begin{document}

\title[Bounds for stochastic integrals indexed by a parameter]{Sharp supremum and H\"older bounds for stochastic integrals indexed by a parameter}

\author{Sonja Cox}
\address{Korteweg-de Vries Institute for Mathematics,
         University of Amsterdam,
         Postbus 94248,
         NL-1090 GE Amsterdam, Netherlands}
\email{s.g.cox@uva.nl}

\author{Joris van Winden}
\address{Delft Institute of Applied Mathematics, 
         Delft University of Technology, 
         Mekelweg 4, 2628 CD Delft, 
         Netherlands}
\email{J.vanWinden@tudelft.nl}


\date{\today}

\begin{abstract}
We provide sharp bounds for the supremum of countably many stochastic convolutions taking values in a $2$-smooth Banach space. As a consequence, we obtain sharp bounds on the modulus of continuity of a family of stochastic integrals indexed by parameter $x\in M$, where $M$ is a metric space with finite doubling dimension.
In particular, we obtain a theory of stochastic integration in H\"older spaces on arbitrary bounded subsets of $\R^d$. This is done by relating the (generalized) H\"older-seminorm associated with a modulus of continuity to a supremum over countably many variables, using a Kolmogorov-type chaining argument. We provide two applications of our results: first, we show long-term bounds for Ornstein--Uhlenbeck processes, and second, we derive novel results regarding the modulus of continuity of the parabolic Anderson model.
\end{abstract}

\maketitle

\setcounter{tocdepth}{1}
\tableofcontents

\section{Introduction}
\subsection{Sharp supremum bounds for countably many stochastic integrals}
It is well-known that if $(\gamma_k)_{k\in \N}$ is a sequence of independent standard Gaussian random variables and $(c_k)_{k\in \N}$ is a non-negative decreasing sequence in $\R$ then we have the following equivalence:
\begin{equation}\label{eq:supGauss}
    \EE[\Big]{ \sup_{k\in \N}\, \abs{c_k \gamma_k} } < \infty \quad \iff \quad 
    \sup_{k\in \N}\, \abs{ c_k \sqrt{\log(k)} } < \infty
\end{equation} (see e.g.~\cite[Proposition 2.4.16 and Theorem 2.4.18]{Talagrand:2014}). In addition, we have the celebrated upper Burkholder--Davis--Gundy inequalities for stochastic integrals: let $H,U$ be separable Hilbert spaces, $(\Omega,\calF,\P,(\calF_t)_{t\geq 0})$ a stochastic basis, $W$ a $U$-cylindrical Wiener process with respect to $(\calF_t)_{t\geq 0}$, and let $\calL_2(U,H)$ denote the Hilbert--Schmidt operators from $U$ to $H$. Then there exists a constant $C\in (0,\infty)$ such that 
\begin{equation}\label{eq:BDG}
\norm[\bigg]{ \,\sup_{t\geq 0}\, \norm[\Big]{\int_{0}^{t} f(s) \rd W(s) }_H }_{L^p(\Omega)}
\leq C \sqrt{p}\,
\norm[\bigg]{ \left(  \int_{0}^{\infty} \norm{f(s)}_{\calL_2(U,H)}^2 \rd s \right)^{1/2} }_{L^p(\Omega)}
\end{equation}
for all $p\in [1,\infty)$ and for every $\calL_2(U,H)$-valued progressive process $(f(t))_{t\geq 0}$; this was proven in~\cite{Davis:1976} for $H=U=\R$ and extends to the Hilbert space setting thanks to~\cite{KallenbergSztencel:1991}; see~\cite{Seidler:2010} and~\cite{MarinelliRockner:2016} for detailed literature reviews on Burkholder--Davis--Gundy inequalities for stochastic integrals. Note that the asymptotic dependence on $p$ in~\eqref{eq:BDG} cannot be improved~\cite{Davis:1976}.

Our first main result combines~\eqref{eq:supGauss} and~\eqref{eq:BDG}:
\begin{ttheorem}\label{thm:supseidler_intro}
For every $p\in [1,\infty)$ and every sequence of progressive $\calL_2(U,H)$-valued stochastic processes $(f_k(t))_{t \geq 0}$, $k\in \N$, it holds that
\begin{equation}\label{eq:supBDG}
\begin{aligned}
&\norm[\bigg]{ \sup_{t \geq 0,\, k \in \N} \norm[\Big]{ \int_{0}^{t} f_k(s) \rd W(s) }_H }_{L^p(\Omega)} \\
&\qquad \leq 10 \,
\norm[\bigg]{ \,\sup_{k\in \N} \sqrt{p + \log(k)} \left(\int_{0}^{\infty} \| f_k(s) \|_{\calL_2(U,H)}^2 \rd s \right)^{1/2} }_{L^p(\Omega)}.
\end{aligned}
\end{equation}
\end{ttheorem}
The actual result (\cref{thm:seidlerplus} below) is more general: $H$ is replaced by a $2$-smooth Banach space $X$, and the stochastic integral is replaced by a stochastic convolution with a contractive $C_0$-semigroup on $X$.
The Burkholder--Davis--Gundy inequalities~\eqref{eq:BDG} were proven in this setting in~\cite[Theorem 2.1]{Seidler:2010}. Note that neither the weights $\sqrt{\log(k)}$ nor the $\sqrt{p}$-dependence in~\eqref{eq:supBDG} can be improved.

The proof of \cref{thm:supseidler_intro} is obtained by using the exponential tail estimates from~\cite[Theorem 5.6]{NeervenVeraar:2020} to obtain an associated `good-$\lambda$ inequality'.
This inequality is extended using a union bound, after which \eqref{eq:supBDG} follows from a well-known lemma of Burkholder.
A similar strategy can be applied to other types of exponential tail estimates: see \cref{thm:pinelisplus} below for a result analogous to~\cref{thm:supseidler_intro} for the Burkholder--Rosenthal inequality for discrete-time martingales. The close relationship between inequalities of the type~\eqref{eq:BDG} (i.e., $L^p$-square-function type estimates for which the $p$-dependence in the inequality is known) and associated exponential tail estimates (of Azuma--Hoeffding type) is well-established in the literature, see e.g.~\cite{Geiss:1997,Hitczenko:1990}.
In particular, sharp estimates for the supremum of sequences of discrete-time martingales have been obtained in~\cite[Theorem 1.5]{Geiss:1997}, however, an extension to continuous time would be nontrivial.

\cref{thm:supseidler_intro} (or rather, its generalization \cref{thm:seidlerplus} below) improves \cite[(2.12) in Proposition 2.7]{NeervenVeraar:2022}, which implies that for every $p\in (0,\infty)$ there exists a constant $C_p\in (0,\infty)$ such that
\begin{equation}\label{eq:supBDGold}
\begin{aligned}
&\norm[\bigg]{ \sup_{t\geq 0,\, 1\leq k \leq n} \norm[\Big]{ \int_{0}^{t} f_k(s) \rd W(s) }_H }_{L^p(\Omega)}
\\
& \qquad \leq C_p \log(n)
\norm[\bigg]{\left( \int_{0}^{\infty} \sup_{1\leq k \leq n} \| f_k(s) \|_{\calL_2(U,H)}^2 \rd s \right)^{1/2} }_{L^p(\Omega)}
\end{aligned}
\end{equation}
for all choices of $n\in \N$ and all progressive $\calL_2(U,H)$-valued $(f_k(t))_{t \geq 0}$, $k\in \{1,\ldots,n\}$
(like \cref{thm:seidlerplus} below, \cite[Proposition 2.7]{NeervenVeraar:2022} actually considers case that $f$ takes values in a 2-smooth Banach space). 
Note that the supremum in~\eqref{eq:supBDGold} is inside the temporal integral and the $\log(n)$ in~\eqref{eq:supBDGold} is outside of the supremum; in particular, the supremum can only be taken over a finite number of integrals. Moreover,~\eqref{eq:supBDG} shows that the $\log(n)$ in~\eqref{eq:supBDGold} can be replaced by a $\sqrt{\log(n)}$ -- in the Hilbert space setting this was also recently observed in~\cite[Proposition 2.3]{KliobaVeraar:2023}. Finally,~\cite{NeervenVeraar:2022} only proves that $C_p \leq C\sqrt{p}$ for $p\in [2,\infty)$ (and not for $p \in [1,\infty)$). In this respect \cref{thm:seidlerplus} even provides a (minor) improvement of the Burkholder--Davis--Gundy inequality~\cite[Theorem 4.1]{NeervenVeraar:2022}. 
Indeed, unlike \cite{NeervenVeraar:2022} and \cite{Seidler:2010},
we do not need to make use of a Burkholder--Rosenthal type inequality, and there is no need to discretize the stochastic integral.
This results in a more straightforward proof from start to finish, allows us to cover the cases $p \in [1,2]$ and $p \in [2,\infty)$ simultaneously, and improves the final constant which is found.

Before we demonstrate the applications of \cref{thm:supseidler_intro} (and its generalization \cref{thm:seidlerplus}), let us mention that the related result~\cite[Proposition 2.3]{KliobaVeraar:2023} was recently used to prove optimal pathwise convergence rates for temporal discretizations of stochastic partial differential equations~\cite[Theorem 1.2]{KliobaVeraar:2023}.

\subsection{Long-time estimates for Ornstein--Uhlenbeck processes}
\label{subsec:OUintro}
Our first and most straightforward application of \cref{thm:supseidler_intro}/\cref{thm:seidlerplus} involves the derivation of long-term bounds for Ornstein--Uhlenbeck processes. 
The following result is a special case of \cref{thm:ou} below:

\begin{ttheorem}\label{thm:ou_intro}
Let $(S(t))_{t\geq 0}$ be a $C_0$-semigroup on $H$ satisfying 
\begin{equation}\label{eq:hyp:stable}
 \norm{S(t)}_{\calL(H)} \leq e^{-at}, \qquad t \geq 0,
\end{equation}
for some $a>0$, let $p \in [1,\infty)$, and let $(f(t))_{t \in [0,T]}$ be a progressive $\calL_2(U,H)$-valued process.
Then we have the estimate
    \begin{equation}
        \label{eq:ouestgood_intro}
        \begin{aligned}
        \norm[\bigg]{\sup_{t \in [0,T]}&\norm[\Big]{\int_0^t S(t-s)f(s) \rd W(s)}_H}_{L^p(\Omega)} \\
        &\leq 18\, D \sqrt{p + \log(1 + a T)}\,a^{-1/2} \norm[\Big]{\sup_{t \in [0,T]} \norm{f(t)}_{\calL_2(U,H)} }_{L^p(\Omega)}.
        \end{aligned}
    \end{equation}
\end{ttheorem}

Long-term estimates of Ornstein--Uhlenbeck processes of the type~\eqref{eq:ouestgood_intro} are relevant whenever one studies the behavior of a stochastic partial differential equation near a stable manifold.
In many cases, one wants to characterize a timescale on which the solutions leave a neighborhood of the manifold.
It is the $T$-dependence of the estimate \eqref{eq:ouestgood_intro} which determines this timescale.
In the finite-dimensional case, estimates are derived in \cite[Chapter 3.1]{Berglund:2006} which are foundational to the work.
In the infinite-dimensional case, the development is much more recent.
Key difficulties are outlined in \cite{Hamster:2020}, where the authors derive a tail estimate in order to show long-term stability of a traveling wave perturbed by noise.
The authors use results relating to Talagrand's generic chaining and entropy bounds to derive a tail estimate with the right asymptotic dependence on $p$ and $T$ (\cite[Proposition 3.1]{Hamster:2020}).
Our proof bypasses this advanced machinery, and additionally gives $L^p(\Omega)$-bounds with good constants.
Moreover, we quantify the dependence of the estimate on the parameter $a$ in \eqref{eq:hyp:stable}.
The inequality \eqref{eq:ouestgood_intro} is optimal in terms of the dependence on $a,p$, and $T$.

\subsection{Stochastic integration in H\"older spaces}
\label{subsec:stochintHolderintro}
Our second, more elaborate, application of \cref{thm:supseidler_intro} is to construct a theory of stochastic integration in the H\"older space $C^{\alpha}$.
Recall that $C^{\alpha}$ is neither $2$-smooth nor does it have the UMD property, so that `conventional' vector-valued stochastic integration (as surveyed in \cite{NeerVerWeis:2015}) fails.
Instead, we take a more direct approach by transforming \eqref{eq:supBDG} into an estimate with $C^{\alpha}$ norms.
To accomplish this, recall that for any $f \in C([0,1])$ we have
\begin{equation}\label{eq:Kolchain_1d}
\sup_{\substack{x,y\in [0,1],\\ x\neq y}} \frac{|f(x)-f(y)|}{|x-y|^{\alpha}}
\simeq_{\alpha}
\sup_{n\in \N,\, k\in \{1,\ldots,2^n\}} 2^{\alpha n}\, \abs{f(k2^{-n})-f((k-1)2^{-n})}
\end{equation}
by Kolmogorov's chaining technique (see e.g.~\cite[Appendix A.2]{Talagrand:2014}); here `$X\simeq_{\alpha}Y$' means that there exist constants $c,C\in (0,\infty)$ depending only on $\alpha$ such that $c X \leq Y \leq CX$.
Combining the equivalence \eqref{eq:Kolchain_1d} with \cref{thm:supseidler_intro} we obtain the following:
\begin{ttheorem}\label{thm:parameterstochint}
Let $T > 0$, $p\in [1,\infty)$, $\alpha \in (0,1]$, and let $((\Phi(x)(t))_{t \in [0,T]})_{x \in [0,1]}$ be a family of progressive $\calL_2(U,\R)$-valued processes indexed by $[0,1]$.
Define $I(\Phi)\colon [0,1] \rightarrow L^p(\Omega)$ by 
\begin{equation}
 I(\Phi)(x) = \int_0^{T} \Phi(x)(t)\rd W(t), \quad x \in [0,1],
\end{equation}
and suppose that 
\begin{equation}\label{eq:seidlerplus_Holder_intro_RHS}
K(\Phi) \coloneq 
\left\|
\sup_{x,y \in [0,1],\, x\neq y} 
\frac{  \| \Phi(x) - \Phi(y)\|_{L^2(0,T;\calL_2(U,\R))}}
{(p-\log(|x-y|))^{-1/2}|x-y|^{\alpha}}
\right\|_{L^p(\Omega)}
< \infty.
\end{equation}
Then there exists a continuous modification of $I(\Phi)$ (again denoted by $I(\Phi)$), and we have the estimate
\begin{equation}\label{eq:seidlerplus_Holder_intro_LHS}
\left\| 
\sup_{x,y\in [0,1],\, x\neq y} \frac{ \abs{ I(\Phi)(x) - I(\Phi)(y) } }{|x-y|^{\alpha}} 
\right\|_{L^{p}(\Omega)}
\leq C(\alpha) K(\Phi),
\end{equation}
for some constant $C(\alpha)\in (0,\infty)$ which depends only $\alpha$.
\end{ttheorem}
L\'evy's modulus of continuity theorem ensures that \cref{thm:parameterstochint} is sharp in terms of the moduli of continuity involved in~\eqref{eq:seidlerplus_Holder_intro_RHS} and~\eqref{eq:seidlerplus_Holder_intro_LHS} (see \cref{rem:LevyModCont} below). 
\cref{thm:parameterstochint} is a special case of \cref{thm:seidlerplus_Hoelder} below. Indeed, the latter extends \cref{thm:parameterstochint} in two ways:
\begin{enumerate}
\item Instead of only considering H\"older continuity  in~\eqref{eq:seidlerplus_Holder_intro_LHS}, we consider more general moduli of continuity.
This gives more flexibility, and in particular allows one to move the logarithmic term back and forth between~\eqref{eq:seidlerplus_Holder_intro_RHS} and~\eqref{eq:seidlerplus_Holder_intro_LHS}.
However, we are restricted to moduli of continuity for which an analogue of~\eqref{eq:Kolchain_1d} holds.
In particular, we need $1 < \sup_{x\in (0,\infty)} \frac{w(x)}{w(x/2)} < \infty$ (see \cref{def:admissible_weight} below).
In \cref{sec:genHolder} we provide a systematic study of such admissible moduli of continuity, and discuss relevant examples.
\item Instead of considering stochastic integrands $(\Phi(x))_{x\in [0,1]}$ indexed by $x\in [0,1]$, we consider families $(\Phi(x))_{x\in M}$ indexed by a general metric space $M$. 
This allows us to consider mixed (space-time) regularity of stochastic processes, and is crucial to prove \cref{thm:regPAM_intro} below. We find that the Minkowski and doubling dimensions of $M$ play a role when seeking an appropriate analogue of~\eqref{eq:Kolchain_1d}.
In \cref{sec:goodchaining} we recall these notions of dimension, their relation to the Kolmogorov chaining argument, and show that the dimensional requirements for chaining are satisfied for any bounded subset of Euclidean space.

\end{enumerate}
As the previous discussion suggests, we need to generalize \eqref{eq:Kolchain_1d} as an intermediate step to obtain \cref{thm:seidlerplus_Hoelder}.
This is done in \cref{thm:embedding}, which states that for an admissible modulus of continuity $w$ and a metric space $(M,d_M)$ with Minkowski dimension $d$ and finite doubling dimension, we have
\begin{equation*}
 \sup_{x,y\in M, x\neq y} \frac{\| f(x) - f(y) \|_X}{w(d_M(x,y))}
 \simeq_{w,M} 
 \sup_{n\in \N} \frac{\| f(x_n) - f(y_n) \|_X}{w(n^{-1/d})}
\end{equation*}
for all $f\in C(M,X)$, where $(x_n,y_n)\in M\times M$ is an appropriately chosen sequence (that does not depend on $f$ or $w$). \cref{thm:embedding} is similar in spirit to Ciesielski's embedding for the H\"older spaces~\cite{Ciesielski:1960a}, but the underlying philosophy is slightly different (see \cref{rem:compare_Ciesielski} for a more detailed discussion).
As a byproduct of \cref{thm:embedding}, we also obtain an elegant short proof of the Kolmogorov--Chentsov theorem, see~\cref{thm:Kolmogorov-Chentsov} and \cref{rem:compareKU_KC}.

In \cref{sec:regPAM} we use \cref{thm:seidlerplus_Hoelder} (i.e., the generalization of \cref{thm:parameterstochint}) to investigate the regularity of the 1D parabolic Anderson model (i.e., the stochastic heat equation with linear multiplicative noise). More specifically, we prove the following: 
\begin{ttheorem}\label{thm:regPAM_intro}
Let $U\colon [0,T]\times [0,1]\times \Omega \rightarrow \R$ be the mild solution to 
\begin{align*}
    \tfrac{\partial}{\partial t}U(t,x) &= \tfrac{\partial^2}{\partial x^2} U(t,x) + U(t,x)\,
    \xi(t,x), & & (t,x) \in (0,T) \times (0,1), \\
    U(t,0) &=U(t,1)  = 0, & &  t \in (0,T],\\
    U(0,x)& = U_0(x), & & x \in [0,1],
\end{align*}
where $\xi$ is space-time white noise and $U_0\in L^p(\Omega;C^2([0,1]))$ for some $p\in (4,\infty)$. 
Then we have
\begin{equation*}
\left\| \sup_{\substack{t,s\in [0,T];\\ x,y\in [0,1];\\ (t,x)\neq (s,y)}} \frac{|U(t,x)-U(s,y)|}{(1-\tfrac{1}{4}\log(|t-s|))^{\frac{1}{2}} |t-s|^{\frac{1}{4}} + (1-\tfrac{1}{2}\log(|x-y|))  |x-y|^{\frac{1}{2}}}\right\|_{L^p(\Omega)} < \infty.
\end{equation*}
\end{ttheorem}
This improves the classical regularity results for the parabolic Anderson model, which typically use Sobolev embeddings or the Kolmogorov--Chentsov theorem to prove that $U \in L^p(\Omega;C^{\alpha-\frac{1}{p},2(\alpha-\frac{1}{p})}([0,T]\times [0,1]))$ for all $\alpha \in (\frac{1}{p},\frac{1}{4})$, and then a localisation argument to prove that $U\in C^{\alpha,2\alpha}([0,T]\times [0,1])$ a.s.\  for all $\alpha\in (0,\frac{1}{4})$ (see e.g.~\cite{vanNeervenVeraarWeis:2007}). Our technique circumvents the use of the Kolmogorov--Chentsov theorem (whence we do not lose $\frac{1}{p}$ in regularity) \emph{and} considers a stronger modulus of continuity.  The technique for proving \cref{thm:regPAM_intro} essentially combines \cref{thm:seidlerplus_Hoelder} (the generalization of \cref{thm:parameterstochint}) with regularity results for the Dirichlet Green's function. We note that this approach can also be employed to obtain analogous results for other stochastic differential equations involving a linear second-order differential operator in the leading term.

\subsection{Acknowledgments}
In a preliminary stage of this project we benefited from fruitful discussions with Stefan Geiss. In particular, Stefan Geiss kindly pointed out \cref{prop:holder_logblowup} to us, and pointed out the relation between \cref{thm:embedding} and Ciesielski's theorem (see \cref{rem:compare_Ciesielski}).
This work also benefited from many useful comments from Mark Veraar.

Sonja Cox is supported by the NWO grant VI.Vidi.213.070. Joris van Winden is supported by a DIAM fast-track scholarship.

\section{Preliminaries and notation}
\label{sec:prelim}
We use the convention that $\N$ is the set of strictly positive integers and $\N_0 = \N \cup \cur{ 0}$.
Given a finite set $A$, we let $\card(A) \in \N_0$ denote its cardinality. We denote the natural logarithm by $\log$.
When $(S,\calA,\mu)$ is a $\sigma$-finite measure space and $(X,\norm{\cdot}_X)$ is a Banach space, we write $L^p(S;X)$ for the usual Lebesgue--Bochner space of strongly $\calA$-measurable and $p$-integrable (or essentially bounded in the case $p=\infty$) $X$-valued functions on $S$.
Note that if $X$ is separable, the notions of measurability and strong measurability coincide.
In the case where $S$ is an interval (i.e., $S = (a,b)$), we write $L^p(a,b;X)$ instead of $L^p((a,b);X)$.

\subsection{Metric spaces}\label{subsec:metricspacenot}
Let $(M,d_M)$ be a metric space. 
As a matter of convenience we will assume all metric spaces mentioned are nontrivial.
We define the diameter $\Delta(M)\in [0,\infty]$ of $M$ by
\begin{equation*}
\Delta(M)=\sup_{x,y\in M}d_M(x,y),
\end{equation*}
and we may write $\Delta_{d_M}(M)$ to emphasize the dependence on $d_M$ when confusion is possible.
For $x\in M$ and $A\subseteq M$ we set $d_M(x,A)=\inf\{d_M(x,y) : y \in A\}$. For $x\in M$ and $r\in (0,\infty)$ we define the open ball $B_x(r)$ around $x$ with radius $r$ as
\begin{equation*}
B_x(r) = \{y\in M : d_{M}(x,y)<r\}.
\end{equation*}
Similarly, the closed ball $\overline{B}_x(r)$ around $x$ with radius $r$ is defined as
\begin{equation*}
    \overline{B}_x(r) = \cur{ y\in M : d_{M}(x,y) \leq r }.
\end{equation*}
We may write $B_x(d_M,r)$ (resp.\ $\overline{B}_x(d_M,r)$) to emphasize the dependence on $d_M$ when confusion is possible.
For $\eta \in (0,\infty)$, we let $\calN(M,d_M,\eta)$ denote the minimal number of open $d_M$-balls of radius $\eta$ needed to cover $M$, i.e.\
\begin{equation*}
    \calN(M,d_M,\eta) = \min \cur[\Big]{\card(A) : A \subseteq M \text{ finite},\, M \subseteq \bigcup_{x \in A} B_x(\eta)},
\end{equation*}
where we use the convention that the minimum of the empty set is $\infty$.

\subsection{Functional analysis}
We recall the definition of a $p$-smooth Banach space; for details see e.g.~\cite[Section 2.2]{NeervenVeraar:2020},~\cite[Chapter 3]{Woyczynski:2019}. Note that the classical definition involving the modulus of smoothness of $X$ (see, e.g.,~\cite[Definition 3.1.2]{Woyczynski:2019}) is equivalent to this definition (see~\cite[Proposition 3.1.2]{Woyczynski:2019}):

\begin{definition}\label{eq:psmooth}
Let $p\in [1,2]$ and $D\geq 1$. A Banach space $(X, \norm{\cdot}_X)$ is called \emph{$(p,D)$-smooth} if 
\begin{equation*}
\forall x,y \in X \colon 
\norm{x+y}^p_X + \norm{x - y}^p_X \leq 2 \norm{x}_X^p + 2 D^p \norm{y}_X^p,
\end{equation*}
and it is called $p$-smooth if it is $(p,D)$-smooth for some $D\geq 1.$
\end{definition}

\begin{remark}
By the parallelogram identity every Hilbert space is $(2,1)$-smooth. Moreover, for $2\leq p <\infty$ the Lebesgue space $L^p(S)$ is $(2,\sqrt{p-1})$-smooth, see~\cite[Proposition 2.1]{Pinelis:1994} and~\cite[Proposition 2.2]{NeervenVeraar:2022}. 
\end{remark}

\begin{definition}
A family $(S(t))_{t \geq 0}$ of bounded linear operators on a Banach space $(X,\norm{\cdot}_X)$ is called a \emph{$C_0$-semigroup} if it satisfies the following conditions:
\begin{enumerate}
\item $S(0) = I$,
\item $S(t+s) = S(t)S(s)$ for all $t,s\geq 0$,
\item the map $t \mapsto S(t)$ is strongly continuous on $[0,\infty)$.
\end{enumerate}
A $C_0$-semigroup $(S(t))_{t \geq 0}$ is called \emph{contractive} if the inequality $\norm{S(t)x}_X \leq \norm{x}_X$ holds for all $t \geq 0$ and all $x \in X$.
\end{definition}

We also recall the notion of $\gamma$-radonifying operators. This class of operators generalizes Hilbert--Schmidt operators in the context of Banach-space valued stochastic integration.
\begin{definition}
    \label{def:gammaradon}
    Let $H$ be a separable Hilbert space and let $(X,\norm{\cdot}_X)$ be a Banach space.
    The space of \emph{$\gamma$-radonifying operators from $H$ to $X$}, denoted $\gamma(H,X)$, consists of the closure of the finite rank operators $T \colon H \to X$ with respect to the norm
    \begin{equation}
        \norm{T}_{\gamma(H,X)}
        \coloneq \EE[\bigg]{\norm[\Big]{\sum_{i\in \N} \gamma_i\, T h_i}^2_X }^{1/2},
    \end{equation}
    where $(\gamma_i)_{i \in \N}$ is a sequence of independent standard Gaussian variables, and $(h_i)_{i \in \N}$ is an orthonormal basis of $H$.
\end{definition}

We remark that the definition of $\gamma(H,X)$ is independent of the choice of basis, that $\gamma(H,X)\simeq \mathcal{L}_2(H,X)$ when $X$ is a Hilbert space, and that $\gamma(H,X) \hookrightarrow \calL(H,X)$.
See also \cite{Neerven:2010} and \cite[Chapter 9]{HytonenEtAl:2017} for a detailed treatment of $\gamma$-radonifying operators.

\subsection{Stochastic calculus}\label{ssec:stochcal}
We fix once and for all a filtered probability space $(\Omega, \calF,\P, (\calF_t)_{t \geq 0})$.
It is implied throughout the paper that all mentioned random variables and processes live on this space, unless the contrary is explicitly stated.
Moreover, when we speak of adaptedness or progressive measurability without mentioning a filtration, it is with respect to $(\calF_t)_{t \geq 0}$.
We also fix a separable Hilbert space $(H,\langle \cdot,\cdot \rangle_H)$, which is to be used for constructing our Wiener processes.

We will make frequent use of the theory of stochastic integration theory in $2$-smooth Banach spaces.
We do not give a full introduction here (instead referring to \cite{NeerVerWeis:2015}),
but just recall that this allows us to define a stochastic integral of the form
\begin{equation}
    \label{eq:defstochint}
    I(f) \coloneq \int_0^T f(s) \rd W(s),
\end{equation}
when $T > 0$, $(X,\norm{\cdot}_X)$ is a separable $2$-smooth Banach space, $W$ is an $H$-cylindrical Wiener process, and $(f(t))_{t \in [0,T]}$ is a progressive $\gamma(H,X)$-valued process (see \cref{def:gammaradon}) with sample paths in $L^2(0,T;\gamma(H,X))$ a.s.
In the finite-dimensional case, this notion of stochastic integration reduces to the usual It\^o integral.

If additionally $(S(t))_{t \geq 0}$ is a contractive $C_0$-semigroup on $X$, it is well-known (see \cite[Section 5]{NeervenVeraar:2020} for an overview of the historical development) that the convolution process $(\Psi(t))_{t \in [0,T]}$, defined via
\begin{equation}
    \label{eq:defstochconv}
    \Psi(t) \coloneq \int_0^t S(t-s)f(s) \rd W(s), \quad t \in [0,T],
\end{equation}
has a version with sample paths in $C([0,T],X)$ a.s.
Throughout the paper, when we write stochastic convolutions of the form \eqref{eq:defstochconv}, it is thus implied that we use such a continuous version.

\section{Optimal bounds for the supremum of multiple stochastic processes}
In this section, we prove a $\sqrt{\log(n)}$-weighted bound for the moments of the supremum of countably many `martingale-like' processes. More specifically, in \cref{thm:seidlerplus} we consider sequences of stochastic convolutions, and \cref{thm:pinelisplus} we consider sequences of discrete-time martingales. 
The general philosophy behind the proof is that an exponential tail estimate for a `martingale-like' quantity implies a weighted $L^p(\Omega)$ estimate for sequences of this quantity. We take the tail estimates in \cite[Theorem 5.6]{NeervenVeraar:2020} as an input for the stochastic convolutions, and those in \cite[Lemma 4.2]{Pinelis:1994} as an input for the discrete-time martingales. Note that this approach can be applied to other exponential tail estimates for  `martingale-like' quantities, see also \cref{rem:pinelisplusBDG} below.

\subsection{Optimal Burkholder--Davis--Gundy type bounds for the supremum of multiple stochastic convolutions}

Throughout the section, we let $T > 0$ and let $(X,\norm{\cdot}_X)$ be a $(2,D)$-smooth Banach space.
We also let $(W_j)_{j \in \N_0}$ be a sequence of (not necessarily independent) $H$-cylindrical Wiener processes on $(\Omega, \calF,\P, (\calF_t)_{t \geq 0})$ (recall the stochastic setup from \cref{sec:prelim}), and set $W \equiv W_0$.
We can now state the main result of this section.
\begin{theorem}
\label{thm:seidlerplus}
Let $p \in [1,\infty)$, let $((\psi_j(t))_{t \in [0,T]})_{j \in \N}$ be a sequence of progressive $\gamma(H,X)$-valued processes, and let $((S_j(t))_{t \geq 0})_{j \in \N}$ be a sequence of contractive semigroups on $X$.
Then we have the inequality
\begin{equation}
\begin{aligned}
	\label{eq:seidlerplus}
	& \left\|
        \sup_{j \in \N} \sup_{t \in [0,T]} 
        \left\| 
            \int_0^t S_j(t - s)\psi_j(s) \rd W_j(s)
        \right\|_X
    \right\|_{L^p(\Omega)} 
\\ &	\qquad 
  \leq 10 \,D 
  \left\|  
  \sup_{j\in \N} \left( 
    \sqrt{p+\log(j)} 
    \| \psi_j \|_{L^2(0,T;\gamma(H,X))}
    \right)
  \right\|_{L^{p}(\Omega)}.
  \end{aligned}
\end{equation}
\end{theorem}
\begin{remark}
    In \cref{thm:seidlerplus}, we could additionally let $H,X,T$ or even the filtration or probability space depend on the index $j$ without significantly altering the proof.
    However, we refrain from doing so, as it would be detrimental to the presentation.
\end{remark}
As mentioned, the proof of \cref{thm:seidlerplus} hinges on the exponential tail estimate~\cite[Theorem 5.6]{NeervenVeraar:2020}.
For the reader's convenience, we state a version of this theorem which is adapted to our notation and our specific usecase.
\begin{theorem}[Van Neerven, Veraar]\label{thm:NVtailestimate}
Let $(S(t))_{t \geq 0}$ be a contractive $C_0$-semigroup on $X$, and let
$(\psi(t))_{t \in [0,T]}$ be a progressive $\gamma(H,X)$-valued process which satisfies 
\begin{equation*}
    \PP[\Big]{\norm{\psi}_{L^2(0,T;\gamma(H,X))} \leq \sigma} = 1
\end{equation*}
for some $\sigma \in (0,\infty)$.
Then the following exponential tail estimate holds:
\begin{equation}
\begin{aligned}
    &\PP[\Big]{\sup_{t\in [0,T]} \norm[\Big]{ \int_0^t S(t-s)\psi(s)\rd W(s)}_X \geq \lambda} \leq 3 \exp\bra[\Big]{-\frac{\lambda^2}{4D^2 \sigma^2}}, \quad \lambda \geq 0.
\end{aligned}
\end{equation}
\end{theorem}

\begin{remark}
    The original theorem in \cite{NeervenVeraar:2020} is formulated for (contractive) evolution families, a setting in which our results and methods still apply.
    However, to avoid technical conditions and for the sake of presentation, we have restricted ourselves to the semigroup setting.
    Note that this setting still includes `plain' stochastic integrals, since the identity map is trivially a $C_0$-semigroup.
\end{remark}

In order to prove \cref{thm:seidlerplus}, we first derive from \cref{thm:NVtailestimate} a good-$\lambda$ inequality for stochastic convolutions.

\begin{lemma}
    \label{lem:goodlambdastochint}
    Let $(S(t))_{t \geq 0}$ be a contractive $C_0$-semigroup on $X$, let $(\psi(t))_{t \in [0,T]}$ be a progressive $\gamma(H,X)$-valued process, and define
    \begin{equation*}
        \Psi(t) \coloneq \int_0^t S(t-s)\psi(s) \rd W(s), \quad t \in [0,T],
    \end{equation*}
    together with $ \Psi^* \coloneq \sup_{t \in [0,T]} \norm{\Psi(t)}_X$ and $s(\psi)  \coloneq \norm{\psi}_{L^2(0,T;\gamma(H,X))}$.
    Then for all $\beta > 1$ and $\delta,\lambda > 0$, we have the inequality
    \begin{equation}
        \PP{\Psi^* > \beta \lambda,\, s(\psi) \leq \delta \lambda} \leq 3 \exp\bra[\Big]{-\frac{(\beta - 1)^2}{4 D^2 \delta^2}} \PP{\Psi^* > \lambda}.
    \end{equation}
\end{lemma}
\begin{proof}
    We use a classical three stopping times argument. 
    We define
    \begin{align*}
        \mu &\coloneqq \sup\cur{t \in [0,T] : \norm{\Psi(t)}_X \leq \beta \lambda}, \\
        \nu &\coloneqq \sup\cur{t \in [0,T] : \norm{\Psi(t)}_X \leq \lambda}, \\
        \sigma &\coloneqq \sup\cur{t \in [0,T] : \norm{\psi}_{L^2(0,t;\gamma(H,X))} \leq \delta \lambda},
    \end{align*}
    and additionally introduce
    \begin{equation}
        \label{eq:defPhiaux}
        \Phi(t) \coloneqq \int_\nu^{t} S(t-s)1_{[0,\sigma]}(s)\psi(s)\rd W(s), \quad t \in [\nu,T].
    \end{equation}
    together with $\Phi^* \coloneqq \sup_{t \in [\nu,T]} \norm{\Phi(t)}_X$.
    Observe that the event $\cur{s(\psi) \leq \delta \lambda}$ implies $\sigma = T$ so that also
    \begin{equation*}
        \Psi(\mu) = S(\mu-\nu)\Psi(\nu) + \Phi(\mu).
    \end{equation*}
    Since $\norm{S(\mu-\nu)\Psi(\nu)}_X \leq \lambda$ by contractivity and the definition of $\nu$, we see from the reverse triangle inequality that the event $\cur{\Psi^* > \beta \lambda,\, s(\psi) \leq \delta \lambda}$ implies that $\Phi^* \geq (\beta - 1)\lambda$.
    It also trivially implies that $\nu < T$.
    Thus, we have
    \begin{align*}
        \PP{\Psi^* > \beta \lambda,\, s(\psi) \leq \delta \lambda}
        &\leq \PP{\Phi^* > (\beta - 1)\lambda,\, \nu < T} \\
        &= \PP{\Phi^* > (\beta - 1)\lambda \mid \nu < T}\PP{\Psi^* > \lambda}.
    \end{align*}
    It only remains to estimate the conditional probability.
    To do this, we note that since the integral defining $\Phi$ starts at time $\nu$,  \eqref{eq:defPhiaux} can also be interpreted as a stochastic integral with respect to the conditional probability measure $\widetilde{\P}(\cdot) \coloneqq \PP{\cdot \mid \nu < T}$ by the strong Markov property of a Wiener process.
    Thus, we can use \cref{thm:NVtailestimate} to estimate
    \begin{align*}
        \widetilde{\P}\bra[\big]{\Phi^* > (\beta - 1)\lambda}
        \leq 3 \exp\bra[\Big]{-\frac{(\beta - 1)^2}{4 D^2 \delta^2}},
    \end{align*}
    since $\norm{\ind_{[0,\sigma]}\psi}_{L^2(0,T;\gamma(H,X))} \leq \delta \lambda$ by definition of $\sigma$.
\end{proof}
\begin{proof}[Proof of \cref{thm:seidlerplus}]
    For each $j \in \N$, we define 
    \begin{align*}
        \Psi_j^* \coloneqq \sup_{t \in [0,T]} \norm[\Big]{\int_0^t S_j(t-s) \psi_j(s) \rd W(s)}_X, 
        \quad s(\psi_j) \coloneqq \norm{\psi_j}_{L^2(0,T;\gamma(H,X))},
    \end{align*}
    together with the parameters
    \begin{align*}
        \beta = 2, 
        \quad \delta_j \coloneqq \frac{1}{4D\sqrt{p + \log(j)}}, 
        \quad \eps_j \coloneqq 3\exp\bra[\Big]{-\frac{(\beta - 1)^2}{4D^2 \delta_j^2}}.
    \end{align*}
    By a union bound and \cref{lem:goodlambdastochint}, we find the inequality
    \begin{align*}
        \PP{\sup_{j\in \N} \Psi^*_j > \beta& \lambda,\, \sup_{j\in\N} \delta_j^{-1}s(\psi_j) \leq \lambda} 
        \leq \sum_{j\in \N} \PP{\Psi^*_j > \beta \lambda,\, \delta_j^{-1}s(\psi_j) \leq \lambda} \\
        &\leq \sum_{j\in \N} \eps_j \PP{\Psi^*_j > \lambda}
        \leq \bra[\Big]{\sum_{j\in \N} \eps_j} \, \PP{\sup_{j\in\N} \Psi^*_j > \lambda},
    \end{align*}
    which is valid for all $\lambda > 0$.
    We now observe that $\eps_j = 3e^{-4p}j^{-4}$, which implies
    \begin{equation*}
        \label{eq:betaepsestimate}
        \beta^p \eps_j = 3 (\beta e^{-4})^p j^{-4} \leq 3 \beta e^{-4} j^{-4}, \quad j \in \N,
    \end{equation*}
    since $p \geq 1$ and $\beta e^{-4} = 2e^{-4} \leq 1$. Therefore, we have
    \begin{equation*}
         \sum_{j\in\N} \beta^p \eps_j \leq 3 \beta e^{-4}\sum_{j\in\N} j^{-4} = \tfrac{1}{15}\pi^4 e^{-4} < \tfrac{3}{25},
    \end{equation*}
    as can be numerically verified.
    Thus, we may apply \cite[Lemma 7.1]{Burkholder:1973} with $\Phi(x) = x^p$, $\gamma = 2^p$, $\eps = \sum_{j \in \N} \eps_j$, and $\delta = \eta = 1$ to find
    \begin{equation*}
        \EE[\Big]{\bra[\big]{\sup_{j\in\N} \Psi^*_j}^p} 
        \leq 2^p \tfrac{25}{22}\EE[\Big]{\bra[\big]{\sup_{j\in\N} \,\delta_j^{-1}s(\psi_j)}^p}.
    \end{equation*}
    Substituting the definition of $\delta_j$ and using $\tfrac{25}{22}\leq \tfrac{10}{8} \leq (\tfrac{10}{8})^p$ gives
    \begin{equation*}
        \EE[\Big]{\bra[\big]{\sup_{j \in \N} \Psi^*_j}^p} \leq \bra{10\, D}^p 
        \EE[\Big]{\bra[\big]{\sup_{j \in \N} \sqrt{p + \log(j)}s(\psi_j)}^p},
    \end{equation*}
    which results in the desired estimate upon taking $p$-th roots.
\end{proof}
\subsection{Optimal Burkholder--Rosenthal type bounds for the supremum of multiple martingales}
Using the same technique as in the previous section, we now extend Pinelis' Burkholder--Rosenthal inequality \cite[Theorem 4.1]{Pinelis:1994} to sequences of martingales.
We first introduce some notation.
For any $X$-valued martingale $(f(i))_{i \in \N_0}$ with respect to the filtration $(\calF_i)_{i \in \N_0}$, we define the quantities:
\begin{subequations}
\begin{align*}
    f^* &\coloneqq \sup_{i \in \N_0}\,\norm{f(i)}_X, \\
    d^*(f) &\coloneqq \sup_{i \in \N}\,\norm{f(i) - f(i-1)}_X, \\
    s(f) &\coloneqq \bra[\Big]{\sum_{i \in \N} \EE[\Big]{\norm{f(i) - f(i-1) }^2_X \,\Big|\, \calF_{i-1}}}^{\frac{1}{2}}.
\end{align*}
\end{subequations}
The notation is chosen to remain consistent with \cite{Pinelis:1994} whenever possible. We also recall that a martingale $(f(i))_{i\in \N_0}$ with respect to a filtration $(\calF_i)_{i\in \N_0}$ is said to have \emph{conditionally symmetric increments} if $\PP{f(i) - f(i-1)\in B \,|\, \calF_{i-1}}=\PP{f(i-1)-f(i)\in B\,|\, \calF_{i-1}}$ for all Borel sets $B\subseteq X$ and all $i\in \N$.
\begin{theorem}
\label{thm:pinelisplus}
Let $p \in [1,\infty)$, and let $((f_j(i))_{i \in \N_0})_{j \in \N}$ be a sequence of $X$-valued processes, each of which is a martingale with respect to $(\calF_i)_{i \in \N_0}$, starts at zero, and has conditionally symmetric increments.
Then we have the inequality
\begin{equation}
\begin{aligned}
	\label{eq:pinelisplus}
	\norm[\big]{ \sup_{j \in \N} f^*_j }_{L^p(\Omega)} 
	&\leq 13\,\norm[\big]{ \sup_{j \in \N}\, (p + \log(j))\, d^*(f_j) }_{L^p(\Omega)} 
		\\
    &\quad+ 14\, D\, \norm[\big]{ \sup_{j \in \N}\sqrt{p + \log(j)}\,s(f_j) }_{L^p(\Omega)}.
\end{aligned}
\end{equation}
\end{theorem}

\begin{proof}
    We set $\beta = 2 + \frac{1}{5}$, and introduce for $j \in \N$ the quantities
    \begin{align*}
         \delta_{j,1} &\coloneqq \frac{2}{11\sqrt{p+\log(j)}}, &
         \delta_{j,2} &\coloneqq \frac{1}{5(p + \log(j))}, \\
		N_j &\coloneqq \frac{\beta - 1 - \delta_{j,2}}{\delta_{j,2}}, &
		 \eps_j &\coloneqq 2 \bra[\Big]{\frac{e}{N_j}\frac{\delta_{j,1}^2}{\delta_{j,2}^2}}^{N_j},
    \end{align*}
    as well as
    \begin{equation*}
        w^*(f_j) \coloneqq \max \cur[\Big]{\delta_{j,2}^{-1}\cdot d^*(f_j),\;  D\,\delta_{j,1}^{-1} \cdot s(f_j)}.
    \end{equation*} 
    We begin by observing that
    \begin{align}\label{eq:beta_satisfies_cond}
        \beta - 1 - \delta_{j,2} \geq 1 > 0, \quad j \in \N.
    \end{align}
    Thus, by \cite[Lemma 4.2]{Pinelis:1994}, we have the inequality
    \begin{equation*}
        \PP{f^*_j > \beta \lambda, \, w^*(f_j) \leq \lambda } \leq \eps_j \PP{f^*_j > \lambda}
    \end{equation*}
    for all $\lambda > 0$ and $j \in \N$.
    By the same type of union bound used in the proof of \cref{thm:seidlerplus}, this implies
    \begin{equation*}
        \PP{\sup_{j\in\N} f^*_j > \beta \lambda,\; \sup_{j\in\N} w^*(f_j) \leq \lambda } \leq \bra[\Big]{\sum_{j\in\N}\eps_j}\, \PP{\sup_{j\in\N} f^*_j > \lambda}, \quad \lambda > 0.   
    \end{equation*}
    From~\eqref{eq:beta_satisfies_cond}, we also obtain $N_j \geq 5(p + \log(j))$ for every $j \in \N$, which implies
    \begin{equation*}
        \eps_j = 2\bra[\Big]{\frac{e}{N_j}\frac{\delta_{j,1}^2}{\delta_{j,2}^2}}^{N_j}
        \leq 2\bra[\Big]{\frac{20e}{121}}^{N_j}
        \leq 2\, e^{-4p}j^{-4}, \quad j \in \N,
    \end{equation*}
    since $\frac{20e}{121} < e^{-\frac{4}{5}}$, as may be numerically verified.
    This further implies
    \begin{equation*}
        \beta^p \eps_j \leq 2(\beta e^{-4})^p j^{-4} \leq 2\beta e^{-4}j^{-4}, \quad j \in \N,
    \end{equation*}
    so that
    \begin{equation*}
        \sum_{j\in\N} \beta^p \eps_j \leq 2 \beta e^{-4} \sum_{j \in \N} j^{-4} = 2 \beta e^{-4}\tfrac{\pi^4}{90} < \tfrac{1}{11}.
    \end{equation*}
    We now apply \cite[Lemma 7.1]{Burkholder:1973} with $\Phi(x) = x^p$, $\gamma = \beta^p$, and $\delta = \eta = 1$ to obtain
\begin{align*}
    \EE[\big]{(\sup_{j\in\N} f^*_j)^p} \leq \beta^p \tfrac{11}{10}\EE[\big]{(\sup_{j\in\N} w^*(f_j))^p}.
\end{align*}
Taking $p$-th roots and using the definition of $w^*(f_j)$ and $\beta$ gives
\begin{equation*}
\begin{aligned}
    \norm{\,\sup_{j\in\N} f^*_j}_{L^p(\Omega)} 
    \leq \tfrac{11}{5} \tfrac{11}{10}\bra[\Big]{
    & 5\norm{\,\sup_{j \in \N}\,(p + \log(j))\, d^*(f_j)}_{L^p(\Omega)} \\
    &+ \tfrac{11}{2}D\, \norm{\,\sup_{j \in \N}\sqrt{p + \log(j)}\,s(f_j)}_{L^p(\Omega)}
    },
\end{aligned}
\end{equation*}
which shows the result upon estimating the fractions.
\end{proof}

\begin{remark}\label{rem:pinelisplusBDG}
As mentioned above, the proofs of \cref{thm:seidlerplus} and \cref{thm:pinelisplus} rely on tail bounds for martingales. In particular, analogous results could be obtained for other discrete-time martingale inequalities by employing the tail bounds provided by~\cite[Theorem 3.6]{Pinelis:1994},~\cite[Theorem 1.5]{Naor:2012}, or~\cite[Theorem 1.3]{Luo:2022}.
\end{remark}

\section{Application: long-term bounds for Ornstein--Uhlenbeck processes}
\label{sec:ou}
In this section, we derive long-term estimates for the running maximum of an Ornstein--Uhlenbeck process as direct application of \cref{thm:seidlerplus}.
Throughout the section, let $T > 0$, let $(X,\norm{\cdot}_X)$ be a $(2,D)$-smooth Banach space, and let $(S(t))_{t \geq 0}$ be a contractive $C_0$-semigroup on $X$ which has generator $A$.
We also let $W$ be an $H$-cylindrical Wiener process on $(\Omega, \calF,\P, (\calF_t)_{t \geq 0})$ (recall the stochastic setup from \cref{sec:prelim}). We consider the stochastic evolution equation
\begin{equation}
    \label{eq:infinitedimou}
    \rd u(t) = A u(t) \rd t + f(t) \rd W(t),\quad t\in [0,\infty),
\end{equation}
with $u(0) = 0$, where $(f(t))_{t \in [0,\infty)}$ is a progressive a.s.\ square-integrable $\gamma(H,X)$-valued process. Recall that the (mild) solution to~\eqref{eq:infinitedimou} is commonly referred to as an \emph{Ornstein--Uhlenbeck process}; it is a continuous $X$-valued process satisfying the following variation-of-constants formula:
\begin{equation}\label{eq:ou}
    u(t) = \int_0^t S(t-s)f(s) \rd W(s),\quad t \in [0,\infty)
\end{equation}
(see also \cref{ssec:stochcal}).

We are interested in estimating $u$ in terms of $f$ in the case where the dynamics of $A$ drive the solution to zero at an exponential rate.
This is captured by the following assumption:
\begin{hypothesis}
    \label{hyp:stable}
    The semigroup $(S(t))_{t \geq 0}$ is \emph{exponentially stable}, meaning there exists a constant $a > 0$ such that
    \begin{equation}
        \label{eq:expstab}
        \norm{S(t)}_{\mathcal{L}(X)} \leq e^{-at}, \quad t \geq 0.
    \end{equation}
\end{hypothesis}

Under \cref{hyp:stable}, there is a delicate balance in \eqref{eq:infinitedimou} between exponential decay to zero due to $A$, and stochastic forcing away from zero due to $f \rd W$. Indeed, using~\eqref{eq:ou}, \cite[Theorem 1.1]{Seidler:2010}, H\"older's inequality, and \cref{hyp:stable}, it is straightforward to derive an estimate of the form
\begin{equation}
    \label{eq:longtermitoest}
    \sup_{t \in [0,T]}
    \norm{ u(t)}_{L^p(\Omega;X)} \leq C D \sqrt{p}\,a^{-1/2} \norm[\Big]{\sup_{t \in [0,T]}\norm{f(t)}_{\gamma(H,X)}}_{L^p(\Omega)},
\end{equation}
where $p \in [1,\infty)$ and $C$ is an absolute constant.
However, the case where the supremum over $t$ is inside the expectation on the left-hand side of \eqref{eq:longtermitoest} is more delicate.
There, it is no longer expected to have a constant which is independent of $T$.
For example, the moments of a one-dimensional Ornstein--Uhlenbeck process (corresponding to $H = X = \R$, $A = -\lambda$ and $f \equiv \sigma$ for some $\lambda,\sigma > 0$), already grow like $\sqrt{\log (T)}$, even though we have $f \in L^{\infty}(\Omega \times \R^+)$ in this case.
Using the factorization method directly (see \cite[Theorem 4.5]{NeervenVeraar:2020} for a general version), it is possible to derive the estimate
\begin{equation}
    \label{eq:longtermfactest}
    \norm[\Big]{\sup_{t \in [0,T]}\norm{u(t)}_X}_{L^p(\Omega)} \leq C D  \sqrt{p\, T}\,\norm[\Big]{\sup_{t \in [0,T]}\norm{f(t)}_{\gamma(H,X)}}_{L^p(\Omega)},
\end{equation}
where $p \in [1,\infty)$ and $C$ is an absolute constant.
However, this method does not incorporate the exponential stability of the semigroup, and it unclear how to do so.
Consequently,~\eqref{eq:longtermfactest} does not produce a constant which has the expected scaling $\sqrt{\log(T)}$, like in the one-dimensional case.

We now state \cref{thm:ou}, which shows for the first time an $L^p(\Omega)$-estimate which has the correct (joint) asymptotic dependence on $a$, $p$, and $T$.
We hope that \cref{thm:ou} may be of use to future authors attempting to generalize the results of \cite{Berglund:2006} to an infinite-dimensional setting.
Note that the proof is straightforward, and relies mostly on two direct applications of \cref{thm:seidlerplus}.

\begin{theorem}
    \label{thm:ou}
    Suppose that \cref{hyp:stable} holds. 
    Let $p \in [1,\infty)$, and let $(f(t))_{t \in [0,T]}$ be a progressive $\gamma(H,X)$-valued process.
    Then we have the estimate
    \begin{equation}
        \label{eq:ouestgood}
        \begin{aligned}
        \norm[\bigg]{\sup_{t \in [0,T]}&\norm[\Big]{\int_0^t S(t-s)f(s) \rd W(s)}_X}_{L^p(\Omega)} \\
        &\leq 18\, D \sqrt{p + \log(1 + a T)}\,a^{-1/2} 
        \norm[\Big]{\sup_{t \in [0,T]}\norm{f(t)}_{\gamma(H,X)}}_{L^p(\Omega)}.
        \end{aligned}
    \end{equation}
\end{theorem}

\begin{proof}
    By rescaling time, it suffices to prove the case $a = 1$.
    We first consider the situation $T \in \N$, where we set $\calI = \cur{0, \ldots, T-1}$.
    We define
    \begin{align*}
        u(t) &= \int_0^t S(t-s)f(s) \rd W(s), \quad t \in [0,T], \\
        v_n(t) &= \int_{n}^{n+t}S(n+t-s)f(s) \rd W(s), \quad n \in \calI,\, t \in [0,1].
    \end{align*}
    By the semigroup property and linearity of the stochastic integral, we have the identity
    \begin{equation*}
        u(n+t) = S(t)u(n) + v_n(t), \quad n \in \calI,\,t \in [0,1].
    \end{equation*}
    Thus, by the triangle inequality and contractivity of $(S(t))_{t \geq 0}$, we get
    \begin{subequations}
    \begin{equation}
        \label{eq:outriangle}
        \sup_{t \in [0,T]}\norm{u(t)}_X \leq \sup_{n \in \calI}\,\norm{u(n)}_X + \sup_{n \in \calI}\sup_{t \in [0,1]}\norm{v_n(t)}_X.
    \end{equation}
    A direct application of \cref{thm:seidlerplus} now gives
    \begin{equation*}
        \norm[\big]{\,\sup_{n \in \calI}\, \norm{u(n)}_X}_{L^p(\Omega)} 
        \leq 10 \, D \sqrt{p + \log(T)} 
        \norm[\Big]{\,\sup_{n \in \calI}\, \norm{S(n - \cdot)f(\cdot)}_{L^2(0,n;\gamma(H,X))}}_{L^p(\Omega)},
    \end{equation*}
    at which point we apply H\"older's inequality and \eqref{eq:expstab} (recall that we are treating the case $a = 1$) to find
    \begin{equation}
        \label{eq:ouestun}
        \norm[\big]{\,\sup_{n \in \calI}\, \norm{u(n)}_X}_{L^p(\Omega)} 
        \leq 5\sqrt{2} \, D \sqrt{p + \log(T)} 
        \norm[\Big]{\sup_{t \in [0,T]}\norm{f(t)}_{\gamma(H,X)}}_{L^p(\Omega)}.
    \end{equation}
    In a similar way, \cref{thm:seidlerplus} and H\"older's inequality also give
    \begin{align}
        \norm[\big]{\,\sup_{n \in \calI}\sup_{t \in [0,1]}\norm{v_n(t)}_X}_{L^p(\Omega)}
        &\leq 10\,D \sqrt{p + \log(T)} \norm[\Big]{\,\sup_{n \in \calI}\, \norm{f}_{L^2(n,n+1;\gamma(H,X))}}_{L^p(\Omega)} \nonumber \\
        \label{eq:ouestvnt}        
        &\leq 10\,D \sqrt{p + \log(T)} 
        \norm[\Big]{\sup_{t \in [0,T]}\norm{f(t)}_{\gamma(H,X)}}_{L^p(\Omega)}.
    \end{align}
    \end{subequations}
    Combining \eqref{eq:outriangle}, \eqref{eq:ouestun}, and \eqref{eq:ouestvnt} then yields
    \begin{equation*}
        \norm[\big]{\sup_{t \in [0,T]} \norm{u(t)}_X}_{L^p(\Omega)}
        \leq 18 D \sqrt{p + \log(T)} 
        \norm[\Big]{\sup_{t \in [0,T]}\norm{f(t)}_{\gamma(H,X)}}_{L^p(\Omega)}.
    \end{equation*}
    The general result now follows by rounding $T$ up to the nearest integer.
\end{proof}

\section{Generalized H\"older spaces}\label{sec:genHolder}

In this section we generalize the notion of a H\"older (semi)norm of functions mapping from a metric space $(M,d_M)$ to a Banach space $(X,\norm{\cdot}_X)$ to allow for more general moduli of continuity.
Our results will be formulated for metric spaces $(M,d_M)$ which satisfy $\Delta(M) \in (0,\infty)$.
The condition $\Delta(M) < \infty$ is equivalent to $(M,d_M)$ being bounded, and $\Delta(M) > 0$ is equivalent to $M$ having at least two distinct points.

\begin{definition}
    \label{def:generalized_holder}
    Let $(M,d_M)$ be a metric space satisfying $\Delta(M) \in (0,\infty)$, let $(X,\norm{\cdot}_X)$ be a Banach space, and let $w \colon (0,1] \to (0,\infty)$ be a function which is non-decreasing and satisfies $\lim_{x \downarrow 0} w(x) = 0$.
    The \emph{generalized H\"older seminorm} $\abs{\cdot}_{C_w(M,X)}\colon X^M\rightarrow [0,\infty]$ is defined by
\begin{equation}
    \label{eq:defgenholdersemi}
 \abs{f}_{C_w(M,X)} = \sup_{x,y\in M, x\neq y} \frac{ \| f(x) - f(y) \|_X }{w( d_{M}(x,y)/\Delta(M))},\quad f\in X^M.
\end{equation}
The associated \emph{generalized H\"older norm} $\left\|\cdot\right\|_{C_w(M,X)}\colon X^M\rightarrow [0,\infty]$ is defined by 
\begin{equation}
    \label{eq:defgenholder}
 \| f \|_{C_w(M,X)} 
 = \sup_{x\in M} \|f(x)\|_X
 +| f |_{C_w(M,X)},
 \quad f\in X^M.
\end{equation}
Finally, we define the Banach space $(C_w(M,X),\, \norm{\cdot}_{C_w(M,X)})$ by 
\begin{equation}
C_w(M,X) = \{ f \in C(M,X) : \norm{f}_{C_w(M,X)} < \infty \}.
\end{equation}
\end{definition}
\begin{remark}
    \label{rem:genholder_rescale}
    Both $\abs{\cdot}_{C_w(M,X)}$ and $\norm{\cdot}_{C_w(M,X)}$ are invariant when scaling the metric $d_M$. 
    We will make frequent use of this property to reduce to the case $\Delta(M) = 1$ in proofs.
\end{remark}

\begin{remark}
    \label{rem:genholder_cont}
    From the fact that $\lim_{x\downarrow 0}w(x)=0$ it follows that $f\in X^M$ is uniformly continuous whenever $\abs{f}_{C_w(M,X)} < \infty$.
Moreover, if $\abs{f}_{C_w(M,X)} < \infty$ then the fact that $\Delta(M)<\infty$ implies that $\sup_{x \in M} \norm{f}_X < \infty$, and therefore also $\norm{f}_{C_w(M,X)} < \infty$.
This also implies that if $M_0$ is dense in $(M,d_M)$ and $f \in C_w(M_0,X)$, then $f$ can be uniquely extended to a function $\tilde{f} \in C_w(M,X)$, which furthermore satisfies $\abs{\tilde{f}}_{C_w(M,X)} = \abs{f}_{C_w(M_0,X)}$ in the case where $w$ is continuous.
\end{remark}

\begin{example}\label{example:Holder} 
Let $\alpha \in (0,1]$ and let $w_{\alpha}\colon (0,1]\rightarrow (0,\infty)$ be given by $w_{\alpha}(x)=x^{\alpha}$, $x\in (0,1]$. 
Then $\left| \cdot \right|_{C_{w_{\alpha}}(M,X)}$ simply measures the $\alpha$-H\"older continuity of a function; 
we set $\left|\cdot \right|_{C^{\alpha}(M,X)}\coloneq|\Delta(M)|^{-\alpha}\left| \cdot \right|_{C_{w_{\alpha}}(M,X)}$ and $C^{\alpha}(M,X)\coloneq C_{w_{\alpha}}(M,X)$.
Note the somewhat unconventional definition of $\left\| \cdot \right\|_{C^1(M,X)}$: this measures the Lipschitz constant of a function and \emph{not} the supremum norm of the derivative.
\end{example}

In~\cref{sec:ciesielski} (see~\cref{thm:embedding}) we will prove that for certain metric spaces there exists an embedding $J\colon C(M,X)\rightarrow \ell^{\infty}(X)$ that defines an isomorphism between $C_{w}(M,X)$ and an appropriately weighted subspace of $\ell^{\infty}(X)$, provided $w$ is admissible in the following sense:

\begin{definition}\label{def:admissible_weight}
We call $w\colon(0,1]\to (0,\infty)$ an \emph{admissible modulus of continuity} provided 
that 
\begin{enumerate}
\item $w$ is continuous and non-decreasing,
\item there exists a constant $d_w \in (1,\infty)$ such that
      \begin{equation}\label{eq:def_dw} 
      \inf_{x\in (0,1]} \frac{w(x)}{w(x/2)} = d_w,
      \end{equation}
\item there exists a constant $c_w\in (1,\infty)$ such that 
      \begin{equation}\label{eq:def_cw} 
      \sup_{x\in (0,1]} \frac{w(x)}{w(x/2)} = c_w.
      \end{equation}
\end{enumerate}
We refer to $c_w$ and $d_w$ as the \emph{growth constants} of $w$.
\end{definition}

\begin{remark}\label{rem:rel_growthbound}
Let $w\colon (0,1]\rightarrow \R$ 
be an admissible modulus of continuity. 
Then $\lim_{x\downarrow 0}w(x)=0$ by~\eqref{eq:def_dw} and the fact that $w$ is non-decreasing.
Moreover, the fact that $d_w>1$ implies
\begin{equation}\label{eq:dyadic_weight_sums}
\forall m\in \N_0, x \in (0,1] \colon \sum_{k=m}^{\infty} w\bra[\Big]{\frac{x}{2^k}} 
\leq w\bra[\Big]{\frac{x}{2^m}} \sum_{k=0}^{\infty} d_w^{-k}
= \frac{d_w}{d_w-1} w\bra[\Big]{\frac{x}{2^m}}.
\end{equation}
\end{remark}

\begin{example}\label{example:weight}
We list some examples of admissible moduli of continuity:
\begin{enumerate}
\item\label{example:weight:holder}
The function $w_{\alpha}\colon (0,1]\rightarrow (0,\infty)$ given by $w_{\alpha}(x)=x^{\alpha}$ from~\cref{example:Holder} is an admissible modulus of continuity with growth constants $d_w=c_w=2^{\alpha}$.
\
\item\label{example:weight:rescale} Let $w\colon (0,1]\rightarrow (0,\infty)$ be an admissible modulus of continuity with growth constants $(c_w,d_w)$, and let $\lambda\in (0,\infty)$. Then $\lambda w$ is again an admissible modulus of continuity with the same growth constants $(c_w,d_w)$.
\item\label{example:weight:modinvlog} Let $w\colon (0,1]\rightarrow (0,\infty)$ be an admissible modulus of continuity with growth constants $(c_w,d_w)$, and let $\gamma \in (0,\infty)$, $\beta \in (0,\infty)$.
Then $\tilde{w}\colon (0,1]\rightarrow (0,\infty)$ given by $\tilde{w}(x)= (1-\beta \log(x))^{-\gamma}w(x)$ is an admissible modulus of continuity with growth constants $(c_{\tilde{w}},d_{\tilde{w}})$ satisfying 
\begin{equation*}
 d_w \leq d_{\tilde{w}} \leq c_{\tilde{w}}\leq (1+\beta\log(2))^{\gamma} c_w.
\end{equation*}
 This follows from the fact that $x\mapsto (1-\beta \log(x))^{-\gamma}$ is an increasing positive function for $x\in (0,1]$, and that

\begin{equation*}
\forall x \in (0,1]\colon 
\frac{1}{1+\beta \log(2)} \leq \frac{1-\beta\log(x)}{1-\beta\log(\frac{x}{2})} \leq 1.
\end{equation*} 

\item\label{example:weight:holderinvlog} 
Combining~\eqref{example:weight:holder} and~\eqref{example:weight:modinvlog}, we obtain that if $\alpha \in (0,1]$, $\gamma\in (0,\infty)$, $\beta \in (0,\infty)$, 
then $w\colon (0,1] \rightarrow (0,\infty)$ given by $w(x) = (1 - \beta \log(x))^{-\gamma} x^{\alpha}$ is 
an admissible modulus of continuity, with growth constants $(c_w,d_w)$ satisfying 
\begin{equation*}
    2^{\alpha} \leq d_w \leq c_w \leq 2^{\alpha}(1+\beta \log(2))^{\gamma}.
\end{equation*}
\item\label{example:weight:modlog} Let $w\colon (0,1]\rightarrow (0,\infty)$ be an admissible modulus of continuity with growth constants $(c_w,d_w)$, and let $\gamma \in (0,\infty)$, $\beta \in (0,\log(2)^{-1}(d_w^{1/\gamma} -1))$.
Assume moreover that $x\mapsto (1-\beta \log(x))^{\gamma}w(x)$ is non-decreasing on $(0,1]$. 
Then $\tilde{w}\colon (0,1]\rightarrow (0,\infty)$ given by $\tilde{w}(x)= (1-\beta \log(x))^{\gamma} w(x)$ is an admissible modulus of continuity with growth constants $(c_{\tilde{w}}, d_{\tilde{w}})$ satisfying
\begin{equation*}
    (1+\beta \log(2))^{-\gamma}d_w \leq d_{\tilde{w}} \leq c_{\tilde{w}} \leq c_w.
\end{equation*}
Note that the condition on $\beta$ implies $(1+\beta \log(2))^{-\gamma} d_w > 1$.

\item\label{example:weight:holderlog} Combining~\eqref{example:weight:holder} and~\eqref{example:weight:modlog} we obtain that if $\alpha \in (0,1]$, $\gamma\in (0,\infty)$, and $\beta\in (0,\frac{\alpha}{\gamma})$,
then $w \colon (0,1] \rightarrow (0,\infty)$ given by $w(x) = (1 - \beta \log(x))^{\gamma} x^{\alpha}$ is 
an admissible modulus of continuity with growth constants $(c_w,d_w)$ satisfying 
\begin{equation*}
    (1 + \beta \log(2))^{-\gamma}2^{\alpha} \leq d_w \leq c_w \leq 2^{\alpha}.
\end{equation*}
Indeed, fact that $\beta < \frac{\alpha}{\gamma}$ guarantees that $\beta < \log(2)^{-1}(2^{\alpha / \gamma} - 1)$ and insures that that $w$ is increasing, so \eqref{example:weight:modlog} applies.
\end{enumerate}
\end{example}

\begin{remark}
\label{rem:logmodrestriction}
Note that for $\beta > 0$ and $x\in (0,1]$ we have  
$  \min(1,\beta)\leq  \frac{1-\beta \log(x)}{1-\log(x)}  \leq \max(1,\beta)
$. In particular, although \cref{example:weight}~\eqref{example:weight:modlog} and~\eqref{example:weight:holderlog} involve restrictions on $\beta$, the resulting seminorm $\left| \cdot \right|_{C_w(M,X)}$ is equivalent to the seminorm obtained by taking $\beta = 1$.
\end{remark}

Generalized H\"older spaces are more natural than they may seem at first sight: indeed, when measuring the regularity of a stochastic process, one often encounters the situation that the paths of a process $X\colon [0,T]\times \Omega \rightarrow \R$ lie in $C^{\alpha}([0,T])$ for all $\alpha \in (0,\alpha^*)$, but not in $C^{\alpha^*}([0,T])$. The following proposition, which was pointed out by Stefan Geiss, shows that the generalized H\"older space involving the modulus of continuity $w(x)=(1-\beta\log(x))^{\gamma}x^{\alpha^*}$ measures how fast the $\alpha$-H\"older constant blows up as $\alpha\uparrow \alpha^*$. 

\begin{proposition}
\label{prop:holder_logblowup}
Let $(M,d_M)$ be a metric space satisfying $\Delta(M)\in (0,\infty)$, let $(X,\left\| \cdot \right\|_X)$ be a Banach space, let $\alpha^*\in (0,1)$, $\gamma\in (0,\infty)$, $\beta \in (0,\tfrac{\alpha^*}{\gamma})$, and let $w\colon (0,1]\rightarrow (0,\infty)$ be given by 
\begin{equation*}
w(x) = (1-\beta \log(x))^{\gamma} x^{\alpha^*}, \quad x \in (0,1].
\end{equation*}
Then for all $f\in C(M,X)$ it holds that
\begin{equation}\label{eq:weighted_Holder_is_supHolder}
\begin{aligned}
& (e^{-1}\beta \gamma )^{\gamma}
 | f |_{C_w(M,X)}
\\ & \quad \leq
 \sup_{\alpha \in (0,\alpha*)} 
    (\alpha^* - \alpha)^{\gamma} \Delta(M)^{\alpha} | f |_{C^{\alpha}(M,X)}
\\& \quad \leq 
 (\alpha^*+e^{-1} \beta \gamma)^{\gamma}
 | f |_{C_w(M,X)}.
\end{aligned}
\end{equation}
\end{proposition}

\begin{proof}
    By scaling the metric, we may reduce to the case $\Delta(M) = 1$.
    By definition of $C^{\alpha}(M,X)$ and $C_w(M,X)$, it then suffices to prove 
    \begin{equation*}\forall x\in (0,1] \colon
        \frac{(e^{-1}\beta \gamma)^{\gamma}}{x^{\alpha^*}(1-\beta \log(x))^{\gamma}} \leq \sup_{\alpha \in (0,\alpha^*)} \frac{(\alpha^* - \alpha)^{\gamma}}{x^{\alpha}} \leq \frac{(\alpha^* + e^{-1}\beta \gamma)^{\gamma}}{{x^{\alpha^*}(1-\beta \log(x))^{\gamma}}}.
    \end{equation*}
    This is equivalent to
    \begin{equation*}
        \forall x\in (0,1] \colon
        e^{-1}\beta \gamma \leq \sup_{\alpha \in (0,\alpha^*)} (\alpha^* - \alpha) x^{(\alpha^* - \alpha)/\gamma} (1-\beta \log(x)) \leq \alpha^* + e^{-1}\beta \gamma.
    \end{equation*}
    Applying the substitutions $z = -\gamma^{-1}\log(x)$ and $\alpha' = z(\alpha^* - \alpha)$, this is further equivalent to showing
    \begin{equation}
        \label{eq:weightedholderrewrite}
        \forall z\in (0,\infty) \colon
        e^{-1}\beta \gamma 
        \leq  (z^{-1} + \beta \gamma)\sup_{\alpha' \in (0,z\alpha^*)} \alpha' e^{-\alpha'} \leq \alpha^* + e^{-1}\beta \gamma.
    \end{equation}
    We now distinguish between the cases $z \alpha^* \leq 1$ and $z \alpha^* > 1$.
    In the first case, the supremum in \eqref{eq:weightedholderrewrite} is attained at the endpoint $\alpha' = z \alpha^*$.
    Hence, the estimate follows since
    \begin{equation*}
        e^{-1}\beta \gamma \leq e^{-1}\alpha^* \leq (z^{-1} + \beta \gamma)z\alpha^*e^{-z\alpha^*} \leq \alpha^* + e^{-1}\beta \gamma.
    \end{equation*}
    In the second case, the supremum in \eqref{eq:weightedholderrewrite} is attained at the interior point $\alpha' = 1$.
    Thus, the desired estimate follows because
    \begin{equation*}
        e^{-1}\beta \gamma \leq (z^{-1} + \beta \gamma)e^{-1} \leq e^{-1}(\alpha^* + \beta \gamma). \qedhere
    \end{equation*}    
\end{proof}

\section{Minkowski- and doubling dimensions and chaining}\label{sec:goodchaining}
In this section we recall the concepts of Minkowski- and doubling dimension and discuss their relation to chaining. More specifically,
the main result (\cref{prop:chainingdoubling} below) 
shows that if a metric space $(M,d_M)$ has Minkowski dimension $d \in (0,\infty)$ and finite doubling dimension, then one can construct a sequence of graphs which is in some sense `$d$-dimensional', and which accurately encodes the metric structure of $M$.
This is exactly the setup that allows for a $d$-dimensional Kolmogorov-type `chaining' argument,
which we will later use to construct an isomorphism between a generalized H\"older space and a subspace of an (appropriately weighted) $\ell^{\infty}$-space (see \cref{thm:embedding}).

The Minkowski dimension affects the isomorphism in a more drastic way than the doubling dimension.
This is convenient, since we find that it generally holds that the Minkowski dimension remains unchanged under many operations, whereas the doubling dimension might increase.
In particular, we show how the Minkowski and doubling dimensions are affected by isomorphisms (\cref{prop:dimensions_isomorphism}) and by passing to a subset (\cref{prop:dimensions_subspace}).
Finally, in view of our intended applications, we show that every bounded subset of $\R^d$ has Minkowski dimension $d$ and finite doubling dimension, see \cref{cor:dimensions_Euclid}.

Recall that if $(M,d_M)$ is a metric space and $\eta \in (0,\infty)$, then $\mathcal{N}(M,d_M,\eta)$ denotes the minimal number of open $d_M$-balls of radius $\eta$ needed to cover $M$, see also \cref{subsec:metricspacenot} for our notational conventions regarding metric spaces.

\begin{definition}\label{def:Minkowski_dim}
A metric space $(M,d_M)$ which satisfies $\Delta(M) \in (0,\infty)$ has \emph{Minkowski dimension} $d \in (0,\infty)$ 
with \emph{covering constant} $c\in [1,\infty)$ if it satisfies
    \begin{equation}
        \label{eq:coveringbound}
        \mathcal{N}(M,d_M,\eta) \leq c(\Delta(M)/\eta)^{d}, \quad \eta \in \left(0,\Delta(M) \right].
    \end{equation}
\end{definition}

\begin{definition}\label{def:doubling_dim}
    A metric space $(M,d_M)$ has \emph{doubling number} $n_2 \in \N$
if every open ball in $M$ with a given radius $r$ can be covered by $n_2$ open balls of radius $r/2$.
In this case, we say that $(M,d_M)$ has \emph{doubling dimension} $d_2 \coloneq \log_2(n_2)$.
\end{definition}

\begin{remark}
    The condition $\Delta(M) \in (0,\infty)$ in \cref{def:Minkowski_dim} is necessary for \eqref{eq:coveringbound} to be sensible.
\end{remark}

It is easy to see that any bounded metric space with doubling dimension $d_2$ also has Minkowski dimension $d_2$.
The converse does not hold;
in fact, a metric space with Minkowski dimension $d$ might not have a finite doubling dimension at all.
As an example, consider $S \subset \ell^{\infty}(\N)$ given by
\begin{equation*}
    S = \bigcup_{n \in \N} 2^{-n} \cur{e_1, \ldots, e_n}
\end{equation*}
(where $e_n$ denotes the $n$-th element of the standard basis) equipped with the distance inherited by the $\ell^{\infty}(\N)$-norm.

\begin{example}\label{example:dimensions_Chebyshev}
Let $d\in \N$, $R\in (0,\infty)$, $D = [0,R]^d$, and let $d_\Cheb$ be the Chebyshev distance on $D$.
Then $(D,d_\Cheb)$ has Minkowski dimension $d$ with covering constant $2^d$, and doubling number $3^d$.
\end{example}
\begin{proof}
    Let $r_1 > r_2 > 0$.
    Then any open (or closed) ball in $(M,d_\Cheb)$ of radius $r_1$ can be covered by $(r_1/r_2 + 1)^d$ open balls of radius $r_2$.
    Applying this with $r_1 = R$ shows the first claim, and applying it with $r_2 = r_1/2$ shows the second claim.
\end{proof}

The following example shows that the Minkowski dimension can be fractional.

\begin{example}\label{example:dimensions_Sierpinski}
Let $S$ be the Sierpiński triangle, and let $d_{\Euc}$ be the Euclidean metric on $S$.
Then $(S,d_{\Euc})$ has Minkowski dimension $\log_2(3) \approx 1.58$.
This can be seen by using the self-similarity of $S$, which guarantees that $\calN(S,d_{\Euc},\eta/2) \leq 3\, \calN(S,d_{\Euc},\eta)$ for every $\eta \in (0,\Delta(S)]$.
\end{example}

We first show how the Minkowski dimension, covering constant, and doubling number behave under isomorphisms.

\begin{proposition}\label{prop:dimensions_isomorphism}
Let $(M,d_M)$ and $(\hat{M},d_{\hat{M}})$ be isomorphic metric spaces, i.e., there exists a bijection $I\colon M\rightarrow \hat{M}$ and constants $c_{\eqref{eq:dimensions_isomorphism}}, C_{\eqref{eq:dimensions_isomorphism}} > 0$ such that one has 
\begin{equation}\label{eq:dimensions_isomorphism}
c_{\eqref{eq:dimensions_isomorphism}} d_M(x,y)
\leq 
d_{\hat{M}}(I(x),I(y))
\leq 
C_{\eqref{eq:dimensions_isomorphism}} d_M(x,y), \quad x,y \in M.
\end{equation}
If $(M,d_M)$ has Minkowski dimension $d $ with covering constant $c$, then $(\hat{M},d_{\hat{M}})$ has Minkowski dimension $d$ with covering constant $c\cdot c_{\eqref{eq:dimensions_isomorphism}}^{-d}C_{\eqref{eq:dimensions_isomorphism}}^d$. Moreover, if $(M,d_M)$ has doubling number $n_2$, then $(\hat{M},d_{\hat{M}})$ has doubling number $n_2^{1 + \ceil{\log_2(C_{\eqref{eq:dimensions_isomorphism}}/c_{\eqref{eq:dimensions_isomorphism}})}}$. 
\end{proposition}

\begin{proof}
The first statement follows immediately from the fact that $\Delta(M) \leq c_{\eqref{eq:dimensions_isomorphism}}^{-1} \Delta(\hat{M})$
and 
\begin{equation*}
    \mathcal{N}(\hat{M},d_{\hat{M}},\eta) \leq \mathcal{N}(M,d_M, C_{\eqref{eq:dimensions_isomorphism}}^{-1}\eta)
    \leq c\cdot C_{\eqref{eq:dimensions_isomorphism}}^d(\Delta(M)/\eta)^d.
\end{equation*}
The second statement follows by observing that for all $\hat{x}\in \hat{M}$ and $r>0$ we have $B_{\hat{x}}(d_{\hat{M}}, r) \subseteq I(B_{I^{-1}(\hat{x})}(d_M,c_{\eqref{eq:dimensions_isomorphism}}^{-1}r))$, whereas $B_{I(y)}(d_{\hat{M}}, r/2) \subseteq B_y(d_M, C_{\eqref{eq:dimensions_isomorphism}} r/2)$ for all $y\in M$.
\end{proof}

\begin{corollary}\label{cor:dimensions_rescale}
Let $(M,d_M)$ be a metric space with Minkowski dimension $d$ and covering constant $c$, and doubling number $n_2$. Then the metric space $(M,(\Delta(M))^{-1}d_M)$ also has Minkowski dimension $d$ with covering constant $c$, and doubling number $n_2$.
\end{corollary}

The following elementary lemma will be used to control the behavior of Minkowski dimension, covering constant, and doubling dimension under taking subsets.
\begin{lemma}
    \label{lem:subspacecover}
    Let $(M,d_M)$ be a metric space and let $A \subseteq M$, $k \in \N$, and $\eta \in (0,\infty)$.
    If $A$ can be covered by $k$ open balls with centers in $M$ and radius $\eta$, then $A$ can also be covered by $k$ open balls with centers in $A$ and radius $2 \eta$.
\end{lemma}

\begin{proof}
    Let $F \subseteq M$ be such that $\card(F) \leq k$ and $A \subseteq \cup_{x \in F} B_x(\eta)$. Without loss of generality, we may assume $B_x(\eta) \cap A \neq \emptyset$ for any $x \in F$.
    For each $x \in F$, pick a point $y \in B_x(\eta) \cap A$, and denote by $G$ the set of these points.
    Obviously $\card(G) \leq k$, and by the triangle inequality we have $A \subseteq \cup_{x \in F} B_x(\eta) \subseteq \cup_{y \in G} B_y(2\eta)$.
\end{proof}

We can now prove that the Minkowski dimension of a metric space is preserved under taking subsets, and the doubling number is at most squared under taking subsets.
Hence, the doubling dimension can increase by a factor $2$.
\begin{proposition}
    \label{prop:dimensions_subspace}
    Let $(M,d_M)$ be a metric space with Minkowski dimension $d$ and covering constant $c$, and let $A \subseteq M$ satisfy $\Delta(A)>0$.
    Then $(A, d_M|_{A \times A})$ has Minkowski dimension $d$ with covering constant $c\, (2\,\Delta(M)/\Delta(A))^d$. 
    Moreover, if $(M,d_M)$ has doubling number $n_2$ then $(A, d_M|_{A \times A})$ has doubling number $n_2^2$.
\end{proposition}
\begin{proof}
    It follows from \cref{lem:subspacecover} that $\mathcal{N}(A,d_M|_{A \times A},\eta) \leq \mathcal{N}(M,d_M,\eta/2)$.
    Using \eqref{eq:coveringbound} then gives the result on the Minkowski dimension.
    
    Next, fix $x \in A$ and $r \in (0,\infty)$.
    Using \cref{def:doubling_dim} twice, we see $B_x(r)$ can be covered by $n_2^2$ open balls of radius $r/4$ with centers in $M$. Thus, by \cref{lem:subspacecover}, $B_x(r)$ can also be covered by $n_2^2$ balls of radius $r/2$ with centers in $A$.
\end{proof}

From \cref{example:dimensions_Chebyshev,prop:dimensions_isomorphism,prop:dimensions_subspace} we obtain the following result about subsets of $\R^d$ with the Euclidean metric.

\begin{corollary}\label{cor:dimensions_Euclid}
Let $d\in \N$ and let $D\subseteq \R^d$ be a bounded set containing at least two points. 
Let $d_\Euc$ be the Euclidean metric on $D$. 
Then $(D,d_{\operatorname{Euc}})$ has Minkowski dimension $d$ with covering constant $(4 d)^d$, and doubling number $(3^4 d^2)^d$.
\end{corollary}
\begin{proof}
    For $x,y \in \R^d$, we have
    \begin{equation*}
        d_\Cheb(x,y) \leq d_\Euc(x,y) \leq \sqrt{d}\cdot d_\Cheb(x,y).
    \end{equation*}
    Set $R= \Delta_{d_{\operatorname{Euc}}}(D)$. By \cref{example:dimensions_Chebyshev,prop:dimensions_isomorphism} we find $([0,R]^d,d_\Euc)$ has Minkowski dimension $d$ with covering constant $(2\sqrt{d})^d$, and doubling number $3^{d(1+\ceil{\log_2 (\sqrt{d})})}$.
    By \cref{prop:dimensions_subspace} and the fact that $\Delta_\Euc([0,R]^d)/\Delta_\Euc(D) \leq \sqrt{d}$ we find that $(D,d_\Euc)$ has Minkowski dimension $d$ with covering constant $(4d)^d$, and doubling number
    \begin{equation*}
        3^{d(2+2\ceil{\log_2 (\sqrt{d})})} \leq (3^{4 + \log_2(d)})^d
        = (3^4 d^{\log_2(3)})^d \leq (3^4 d^2)^d. \qedhere
    \end{equation*}
\end{proof}

In order to establish that a metric space $(M,d_M)$ with Minkowski dimension $d$ and finite doubling dimension allows for the kind of chaining necessary to prove \cref{thm:embedding}, we need to control the amount of `close neighbors' that a point $x \in M$ can have.
The following lemma takes care of this.

\begin{lemma}
    \label{lem:doublepacking}
    Let $(M,d_M)$ be a metric space with doubling number $n_2$, and let $r \in (0,\infty)$, $k \in \N_0$.
    Let $S \subseteq M$ be contained in an open ball of radius $r$, and such that $d_M(x,y) \geq 2^{-k}r$ whenever $x,y \in S$ and $x \neq y$.
    Then we have $\card(S) \leq n_2^{k+1}$.
\end{lemma}
\begin{proof}
    Let $n$ be the minimum number of open balls of radius $2^{-k-1}r$ required to cover $S$.
    By iterating the doubling property, we see that $n \leq n_2^{k+1}$.
    However, since the points of $S$ are separated by a distance $2^{-k}r$, any open ball of radius $2^{-k-1}r$ can contain at most one point of $S$.
    Therefore, $n \geq \card(S)$.
\end{proof}

Using \cref{lem:doublepacking} and a covering argument, we now prove the fundamental result needed for \cref{thm:embedding}.

\begin{proposition}
    \label{prop:chainingdoubling}
    Let $(M,d_M)$ be a metric space with Minkowski dimension $d$ and covering constant $c$, and doubling number $n_2$.
    Then there exists a sequence $(V_n)_{n\in \N_0}$ of subsets of $M$ such that for all $n\in \N_0$ we have:
\begin{enumerate}
    \item\label{item:Mn_incr} $V_n \subseteq V_{n+1}$,
    \item\label{item:Mn_card} $\card(V_n) \leq c\, 3^d 2^{dn}$,
    \item\label{item:Mn_cover} $d_M(x,V_n) \leq \Delta(M) \cdot 2^{-n}$ for all $x\in M$,
    \item\label{item:Un_card} $\card (\{ \cur{x,y} \subseteq  V_n : d(x,y) \in (0, 3 \cdot 2^{-n} \Delta(M)) \}) \leq c \, 3^d n_2^4\, 2^{dn}$.
\end{enumerate}
\end{proposition}
\begin{proof}
We can assume that $\Delta(M)=1$ by rescaling (see \cref{cor:dimensions_rescale}).
We now claim that it suffices to construct a sequence $(V_n)_{n\in \N_0}$ of subsets of $M$ such that for all $n\in \N_0$ we have:
    \begin{enumerate}[label=(\roman*)]
        \item \label{item:Vn_increase} $V_n \subseteq V_{n+1}$,
        \item \label{item:Vn_card} $\card(V_n) \leq c\,3^d 2^{dn}$,
        \item\label{item:Vn_cover} $d_M(x,V_n) \leq 2^{-n}$ for all $x\in M$,
        \item \label{item:Vn_separated} $d_M(x,y) \geq \frac{2}{3}2^{-n}$ whenever $x,y \in V_n$ and $x \neq y$.
    \end{enumerate}
Indeed,~\ref{item:Vn_separated} and \cref{lem:doublepacking} (with $S=B_x(3\cdot 2^{-n})\cap V_n$ and $k = 3$) imply that for all $x \in V_n$, the number of points in $V_n$ which are strictly within distance $3\cdot 2^{-n}$ of $x$ is bounded by $n_2^4$.
Combining this with~\ref{item:Vn_card} results in~\eqref{item:Un_card}.

We construct the sets $V_n$ inductively.
Pick some $x_0 \in M$ and set $V_0 = \cur{x_0}$.
Then \ref{item:Vn_card}-\ref{item:Vn_separated} are satisfied for $n = 0$.
For $n \in \N$, we find using \cref{def:Minkowski_dim} a set $F_n \subseteq M$ such that $\card(F_n) \leq c \,3^d 2^{dn}$ and $M = \cup_{x \in F_n} B_x(\frac{1}{3}2^{-n})$, and then set
\begin{equation*}
    G_n = \cur[\big]{x \in F_n : d(x,V_{n-1}) \geq \tfrac{2}{3} 2^{-n}}.
\end{equation*}
By the triangle inequality, we see that the set $\cur{B_{x}(\frac{1}{3}2^{-n}): x \in G_n}$ forms a cover of $M \setminus \cup_{x \in V_{n-1}} B_{x}(2^{-n})$.
Moreover, using that $M = \cup_{x \in F_n} B_x(\frac{1}{3}2^{-n})$ and that $d_M(x,y) \geq \frac{4}{3}2^{-n}$ whenever $x,y \in V_{n-1}$ we find that $\card(G_n) \leq c\, 3^d 2^{dn}-\card(V_{n-1})$.
By the (finite) Vitali covering lemma, we obtain a set $H_n \subseteq G_n$ such that the balls $\cur{B_x(2^{-n}): x \in H_n}$ again form a cover of $M \setminus \cup_{x \in V_{n-1}} B_{x}(2^{-n})$, and we additionally have $d(x,y) \geq \frac{2}{3}2^{-n}$ whenever $x,y \in H_n$ and $x \neq y$.
We now set $V_n = V_{n-1} \cup H_n$. Note that by construction $V_n$ satisfies~\ref{item:Vn_increase},~\ref{item:Vn_card}, and~\ref{item:Vn_cover}. 
From the construction of $H_n$ it is clear that~\ref{item:Vn_separated} holds if $x \in H_n$ and/or $y \in H_n$, and if $x,y \in V_{n-1}$ then~\ref{item:Vn_separated} holds by induction.
\end{proof}

\section{A Ciesielski-type embedding based on chaining}\label{sec:ciesielski}

The goal of this section is to show that if $M$ is a metric space with finite Minkowski and doubling dimensions (see \cref{def:Minkowski_dim} and \cref{def:doubling_dim}), then there exists an embedding $C(M,X) \stackrel{J}{\hookrightarrow} \ell^{\infty}(X)$ such that for every admissible modulus of continuity $w$ (see \cref{def:admissible_weight} above) there exist constants $c_{M,w}$ and $C_{M,w}$ such that we have
\begin{equation}\label{eq:embedding}
    c_{M,w}
    \| J f \|_{\ell_{w_d}^{\infty}(X)}
    \leq 
    | f |_{C_w(M,X)}
    \leq 
    C_{M,w}
    \| J f \|_{\ell_{w_d}^{\infty}(X)}.
\end{equation}
for every $f \in C(M,X)$.
Here, $w_d$ is a modification of $w$ involving the Minkowski dimension of $M$. 
The exact statement is given in \cref{thm:embedding} below (see also \cref{cor:embedding_euclid}). 
One can think of these results as a generalization of Ciesielski's embedding~\cite{Ciesielski:1960a}, however, the underlying philosophy is distinctly different, see \cref{rem:compare_Ciesielski}.

\begin{theorem}[Generalized Ciesielski's embedding]
\label{thm:embedding}
Let $(M,d_M)$ be a metric space with Minkowski dimension $d$ and covering constant $c$, and doubling number $n_2$.
Let $w$ be an admissible modulus of continuity with growth constants $(c_w,d_w)$, and let $(X,\norm{\cdot}_X)$ be a Banach space.
Then there exists a sequence $(x_k,y_k)_{k\in \N}$ in $M\times M$ such that for every $f \in C(M,X)$ we have 
\begin{equation}\label{eq:metric_Ciesielski_seminorm}
\begin{aligned}
C_{\eqref{eq:embeddingconstants2}}^{-1} | f |_{C_w(M,X)}\leq 
\sup_{k\in \N} 
\frac{ \| f(x_k) - f(y_k) \|_{X}}
{ w\left( k^{-1/d} \right)} 
\leq  C_{\eqref{eq:embeddingconstants1}} | f |_{C_w(M,X)},
\end{aligned}
\end{equation}
where
\begin{align} 
 \label{eq:embeddingconstants2}
 C_{\eqref{eq:embeddingconstants2}}
 & = 3 c_w d_w (d_w-1)^{-1},
 \\\label{eq:embeddingconstants1}
 C_{\eqref{eq:embeddingconstants1}}& = 
    c_w (c\,3^{2d+1}d^{-1}n_2^4)^{\log_2(c_w)/d}.
\end{align} 
Moreover, the sequence $(x_k,y_k)_{k \in \N}$ can be chosen independently of $w$. 
\end{theorem}

\begin{remark}
    \label{rem:embeddingnoncont}
    If $f \notin C(M,X)$, then \eqref{eq:metric_Ciesielski_seminorm} remains valid if we replace $M$  by $M_0 = \cup_k \cur{x_k,y_k}$ on both sides of the equation ($M_0$ is dense in $M$ by construction).  Similarly,~\eqref{eq:euclid_Ciesielski_seminorm} remains valid for $f\notin C(D,X)$ if we replace $D$ by $D_0 = \cup_k\cur{x_k,y_k}$ on both sides.
\end{remark}

\begin{proof}
By rescaling the metric (see \cref{rem:genholder_rescale,cor:dimensions_rescale}) we can assume without loss of generality that $\Delta(M)=1$.
Let $(V_n)_{n \in \N_0}$ be a sequence of subsets of $M$ obtained from \cref{prop:chainingdoubling}, and set $M_0= \cup_{n\in \N_0}V_n$.
It follows from \cref{prop:chainingdoubling}~\eqref{item:Mn_card}-\eqref{item:Mn_cover} that $M_0$ is countable and dense in $(M,d_M)$.
We also set
\begin{equation*}
    E_n \coloneq \cur[\big] { \{x,y\}\subseteq V_n : d_M(x,y)\in (0,3\cdot 2^{-n}) }, \quad n\in \N_0,
\end{equation*}
and note that $\card(E_n) \leq c\,3^d n_2^4 2^{dn}$ by \cref{prop:chainingdoubling}~\eqref{item:Un_card}.
The sequence $(x_k,y_k)_{k\in \N}$ is to be chosen such that $\cup_{k\in \N} \{x_k,y_k\} = \cup_{n\in \N_0} E_n$, however, we must take care in the numbering.
Therefore, we define a non-decreasing sequence $(\theta(n))_{n\in\N_0}$ in $\N$ such that $\theta(0)=0$ and, for all $n\in \N_0$,
\begin{equation}
\label{eq:deftheta}
\begin{aligned}
    \card(E_{n}) & = \theta(2n+1) - \theta(2n),
    \\    
    \max(2^{d(n+1)}-\card(E_{n}), 0) & \leq \theta(2n+2) - \theta(2n+1).
\end{aligned}
\end{equation}
Fix some (dummy) $x^{*}\in M_0$.
We now pick $(x_k,y_k)_{k\in \N}$ such that 
\begin{align*}
    E_n & = \cur[\big]{ \{ x_{\theta(2n)+1}, y_{\theta(2n)+1} \}, \ldots  , \{ x_{\theta(2n+1)}, y_{\theta(2n+1)} \} },\quad n\in \N_0,
\end{align*}
and we set $(x_k,y_k) = (x^*,x^*)$ whenever there exists an $n\in \N$ such that $\theta(2n+1) < k \leq \theta(2n+2)$.
Note that by \eqref{eq:deftheta} and the estimate on $\card(E_n)$ we have
\begin{equation*}
    2^{d(n+1)} \leq \theta(2n+2) - \theta(2n) \leq c 3^d n_2^4 \cdot 2^{dn}, \quad n \in \N_0,
\end{equation*}
so that a telescoping series argument using $\theta(2n) = \sum_{k=0}^{n-1}\theta(2k+2) - \theta(2k)$ gives
\begin{equation*}
    2^{dn} \leq \theta(2n) \leq c\,3^d n_2^4\frac{2^{dn}-1}{2^d-1} \leq c\,3^{d+1}d^{-1} n_2^4 \cdot 2^{dn}, \quad n \in \N_0.
\end{equation*}
Since $w$ is non-decreasing, this implies that for every $n \in \N_0$ and $\theta(2n) < k \leq \theta(2n+2)$, we have
\begin{equation*}
    w((c3^{2d+1}d^{-1}n_2^4)^{-1/d}\, 3 \cdot 2^{-n}) \leq w(k^{-1/d}) \leq w(2^{-n}).
\end{equation*}
This leads via \eqref{eq:def_cw} to the two-sided estimate
\begin{equation}
    \label{eq:wknest}
    C_{\eqref{eq:embeddingconstants1}}^{-1} w(\min(3\cdot 2^{-n},1)) \leq w(k^{-1/d}) \leq w(2^{-n}),
\end{equation}
which holds whenever $n \in \N_0$ and $\theta(2n) < k \leq \theta(2n+2)$.
In the case $\theta(2n) < k \leq \theta(2n+1)$, we have by construction $d_M(x_k,y_k) \leq \min(3\cdot 2^{-n},1)$, so that \eqref{eq:wknest} implies
\begin{equation*}
    \frac{\norm{f(x_k) - f(y_k)}_X}{w(k^{-1/d})} \leq 
     C_{\eqref{eq:embeddingconstants1}}\frac{\norm{f(x_k) - f(y_k)}_X}{w(d_M(x_k,y_k))} 
     \leq C_{\eqref{eq:embeddingconstants1}} \abs{f}_{C_w(M,X)}.
\end{equation*}
In the case $\theta(2n+1) < k \leq \theta(2n+2)$, the same inequality (without the middle term) trivially holds since $x_k = y_k = x^*$.
We conclude that the second inequality in~\eqref{eq:metric_Ciesielski_seminorm} is valid.

For the first inequality, fix $x,y \in M_0$ with $x \neq y$, and set
\begin{subequations}
\begin{align}
    \label{eq:defnxny}
    n_x &\coloneq \min \cur{ k \in \N_0 : x \in V_k}, 
    & & n_y \coloneq \min \cur{ k \in \N_0 : y \in V_k}, \\
    \label{eq:defn0}
    n_0 &\coloneq \max \cur{ k \in \N_0 : d_M(x,y) < 2^{-k}}. & & 
\end{align}
To ease the notation, we also introduce for $n \in \N_0$ the quantities
\begin{align}
    \label{eq:defK}
    K &\coloneq \sup_{k\in \N}\frac{\| f(x_k) -  f(y_k) \|_{X}}{w(k^{-1/d})}, \\
    \label{eq:defKn}
    K_n &\coloneq \sup_{u,v\in V_n \colon d_M(u,v) < 3\cdot 2^{-n}}\frac{\| f(u) -  f(v) \|_{X}}{w(2^{-n})} \\
    \label{eq:defKn2}
    &= \sup_{\theta(2n)< k \leq \theta(2n+1)}
        \frac{\| f(x_k) -  f(y_k) \|_{X}}
        {w(2^{-n})},
\end{align}
\end{subequations}
where the final equality follows from the construction of $(x_k,y_k)_{k \in \N}$ using \cref{prop:chainingdoubling}-\eqref{item:Mn_incr}.
Estimating \eqref{eq:defKn2} using \eqref{eq:wknest} we see that $K_n \leq K$, a fact which we will use multiple times.

We now proceed to `chaining' $f(x)-f(y)$. 
To this end, let $\phi_n\colon M_0 \rightarrow V_n$, $n\in \N_0$ map every point $z\in M_0$ to (one of) its nearest point(s) in $V_n$.
By \eqref{eq:defnxny} it follows that $\phi_{n_x}(x) = x$ and $\phi_{n_y}(y) = y$.
We thus have
\begin{subequations}
\begin{equation}\label{eq:chaining}
\begin{aligned}
    f(x) - f(y)
    &= 
    f(\phi_{n_x}(x)) - f(\phi_{n_y}(y))
    \\
    &=\quad f ( \phi_{n_0}(x))+
     \sum_{j=n_0+1}^{n_x}
        f (\phi_{j}(x))
        -        
        f ( \phi_{j-1}(x))    
    \\ 
    &\quad-f ( \phi_{n_0}(y))-
    \sum_{j=n_0+1}^{n_y}
        f(\phi_{j}(y))
        -        
        f(\phi_{j-1}(y)),
\end{aligned}
\end{equation}
where the summations are taken to be trivial if the lower index exceeds the higher index.
By definition of $\phi_n$ and \cref{prop:chainingdoubling}-\eqref{item:Mn_cover} we have $d_M(z,\phi_n(z)) \leq 2^{-n}$ for every $z \in M_0$, $n \in \N_0$.
Using the triangle inequality and \eqref{eq:defn0}, this gives
\begin{equation*}
    d_M(\phi_{n_0}(x),\phi_{n_0}(y)) \leq 2\cdot2^{-n_0} + d_M(x,y) < 3\cdot 2^{-n_0}.
\end{equation*}
Combining this with \eqref{eq:def_cw} and \eqref{eq:defKn} results in
\begin{equation}
\label{eq:chaining_low_level_estimate}
\frac{\| f ( \phi_{n_0}(x)) - f ( \phi_{n_0}(y)) \|_X }
{w(2^{-(n_0+1)})}
\leq c_w K_{n_0}
\leq c_w K.
\end{equation}
As for the summations in~\eqref{eq:chaining}, notice that by the triangle inequality we have $d_M(\phi_{j}(x),\phi_{j-1}(x))\leq 2^{-j}+2^{-(j-1)} < 3\cdot 2^{-(j-1)}$ for any $j \in \N_0$. This together with~\eqref{eq:defKn} and~\eqref{eq:dyadic_weight_sums} implies
\begin{equation}\label{eq:chaining_sum_estimate}
\begin{aligned}
    \frac{1}{w(2^{-(n_0+1)})}
    \norm[\Big]{ \sum_{j=n_0+1}^{n_x}
        f (\phi_{j}(x))
        -        
        f ( \phi_{j-1}(x))
    }_X
    &\leq  
    \sum_{j=n_0+1}^{n_x}
    \frac{w(2^{-(j-1)})}{w(2^{-(n_0+1)})} K_{j-1} \\
    &\leq \frac{c_w d_w}{d_w-1} K.
\end{aligned}
\end{equation}
\end{subequations}
Clearly, \eqref{eq:chaining_sum_estimate} also holds if we replace $x$ by $y$.
Finally, we know from \eqref{eq:defn0} that $d_M(x,y) \geq 2^{-(n_0+1)}$.
Thus, combining \eqref{eq:chaining}-\eqref{eq:chaining_sum_estimate} and recalling that $w$ is non-decreasing, we obtain 
\begin{equation*}
    \frac{\| f(x) - f(y) \|_X}
    {w(d_M(x,y))}
    \leq
    \frac{\| f(x) - f(y) \|_X}
    {w(2^{-(n_0+1)})}
    \leq C_{\eqref{eq:embeddingconstants1}} K.
\end{equation*}
Recalling the definition of $K$ \eqref{eq:defK}, this is exactly the estimate we wanted to show.
We conclude that the first estimate of \eqref{eq:metric_Ciesielski_seminorm} holds since $x,y \in M_0$ were arbitrary and $M_0$ is dense in $(M,d_M)$.
\end{proof}

In the Euclidean setting, we obtain from \cref{thm:embedding} and \cref{cor:dimensions_Euclid} the following corollary.
\begin{corollary}[Generalized Ciesielski's embedding, Euclidean case] \label{cor:embedding_euclid}
Let $d\in \N$ and let $D\subseteq \R^d$ be bounded with at least two elements.
Let $w$ be an admissible modulus of continuity with growth constants $(c_w,d_w)$, and let $(X,\norm{\cdot}_X)$ be a Banach space.
Then there exists a sequence $(x_k,y_k)_{k \in \N}$ in $D \times D$ and constants $c_{d,w},C_{d,w}$ such that for every $f \in C(D,X)$ we have
\begin{equation}\label{eq:euclid_Ciesielski_seminorm}
c_{d,w}^{-1} | f |_{C_w(D,X)}\leq 
\sup_{k\in \N} 
\frac{ \| f(x_k) - f(y_k) \|_{X}}
{ w\left( k^{-1/d} \right)} 
\leq  C_{d,w} | f |_{C_w(D,X)}.
\end{equation}
Moreover, the sequence $(x_k,y_k)_{k \in \N}$ can be chosen independently of $w$ and the constants $c_{d,w},C_{d,w}$ depend only on $d$ and $w$.

\end{corollary}

\begin{remark}[Relation to Ciesielski's original embedding]\label{rem:compare_Ciesielski}
Ciesielski proved in~\cite{Ciesielski:1960a} that there is an \emph{isomorphism} between $C^{\alpha}([0,1])$ and a weighted $\ell^{\infty}$ space. Note that in \cref{cor:embedding_euclid} we only obtain an isomorphism from $C^{\alpha}(M)$ into a (non-trivial) subspace of a weighted $\ell^{\infty}$ space. The difference lies in the construction of the embedding: in~\cite{Ciesielski:1960a}, the author constructs $I\colon C([0,1])\rightarrow \ell^{\infty}$ as follows 
\begin{equation}\label{eq:Ciesielski_original}
\begin{aligned}
    I(f)
    & =
    \left(f(1)-f(0),
    f(\tfrac{1}{2})
    -\tfrac{f(1)-f(0)}{2},
    f(\tfrac{1}{4})
    -\tfrac{f(\frac{1}{2})-f(0)}{2},\right.
    \\ & 
    \qquad 
    \left.f(\tfrac{3}{4})
    -\tfrac{f(1)-f(\frac{1}{2})}{2},
    \ldots 
    \right).
\end{aligned}
\end{equation}
The underlying philosophy is that a function $f\in C([0,1])$ can be decomposed using a spline basis $(\phi_{n,k})_{n\in \N_0, 1\leq k\leq 2^{-n}}$, and the contribution of the splines at level $n$ scales like $n^{-\alpha}$ if and only $f\in C^{\alpha}([0,1])$ (this only works for $\alpha \in (0,1)$). 

However, taking $M=[0,1]$ in \cref{thm:embedding} and using $V_0 = \{0\}$ and $V_n = \{ k\cdot 2^{-n+1} \colon k\in \{0,1,\dots,2^{n-1}\} \}$, $n\in \N$, in the proof of \cref{thm:embedding} results in the following embedding:
\begin{equation}\label{eq:Ciesielski_adapted}
\begin{aligned}
    J(f)
    & =
    \left(f(1)-f(0),
    f(\tfrac{1}{2})
    -f(0),
    f(1)-f(\tfrac{1}{2}),
    f(\tfrac{1}{4})
    -f(0),f(\tfrac{1}{2})
    -f(\tfrac{1}{4}),
    \ldots 
    \right).
\end{aligned}
\end{equation}
In particular, $J(f)$ necessarily satisfies infinitely many constraints of the type $J(f)(2)+J(f)(3)=J(f)(1)$. This redundancy is a consequence of the chaining philosophy: for every point $x\in V_n$ one needs to control how much $f(x)$ differs from its `neighbors' at distance at most $3\cdot 2^{-n}\Delta(M)$.

It may be possible to extend Ciesielski's original philosophy to functions on a domain in $\R^d$ using wavelet techniques (wavelets being a generalization of the splines used by Ciesielski).
For example, one could try to adapt the proof of \cite[Theorem 2.23]{Triebel:2008}, where an embedding of a Triebel--Lizorkin space on an arbitrary domain in $\R^d$ into a weighted sequence space is shown.
However, even in the Euclidean setting it is unclear to us whether this would be more efficient.
Furthermore, the general (non-Euclidean) setting of \cref{thm:embedding} seems totally out of reach with these methods.
For a demonstration of the power of our general setting, we refer ahead to the proof of \cref{thm:regPAM}, where we endow the space $[0,T] \times [0,1]$ with a metric which is specifically adapted to the situation at hand (see \eqref{eq:pammetric} below).
\end{remark}

\section{Intermezzo: the Kolmogorov--Chentsov theorem}
The Ciesielski-type embeddings of~\cref{thm:embedding,cor:embedding_euclid} 
allow for a quick proof of the Kolmogorov--Chentsov theorem in a setting closely related to~\cite[Theorem 1.1]{KratschmerUrusov:2022}, see \cref{thm:Kolmogorov-Chentsov} below (for a detailed comparison see \cref{rem:compareKU_KC} below). Although \cref{thm:Kolmogorov-Chentsov} is not a new result, we feel the proof is sufficiently elegant and straightforward to deserve attention. Indeed, the only ingredient needed aside from \cref{thm:embedding} is the following elementary lemma:
\begin{lemma}
\label{lem:KCLplemma}
Let $p\in [1,\infty)$, $\alpha \in (\frac{1}{p},\infty)$, $\beta \in [0,\alpha-\frac{1}{p})$ and let $(\Psi_n)_{n\in \N}$ be a sequence of non-negative real-valued random variables. 
Then we have
\begin{equation}
    \norm[\big]{ \sup_{n\in \N} n^{\beta} \Psi_n }_{L^p(\Omega)}
    \leq \bra[\Big]{\tfrac{\alpha-\beta}{\alpha - 1/p - \beta}}^{1/p}
    \sup_{n\in \N} n^{\alpha} \norm{ \Psi_n }_{L^p(\Omega)}.
\end{equation}
\end{lemma}

\begin{proof}
Note that
\begin{align*}
    \EE[\Big]{ \sup_{n\in \N} n^{\beta p}  \Psi_n ^p}
    &
    \leq 
    \sum_{n\in \N}  n^{\beta p} \EE{  \Psi_n^p }
    \leq \sum_{n\in \N} n^{-(\alpha-\beta)p} \sup_{k\in \N} k^{\alpha p} \EE{\Psi_k^p}
   \\  & 
    \leq (\alpha-\beta)(\alpha - 1/p - \beta)^{-1} 
    \sup_{k\in \N} k^{\alpha p} \EE{\Psi_k^p}.
\end{align*}
The estimate then follows after taking $p$-th roots.
\end{proof}

\begin{theorem}\label{thm:Kolmogorov-Chentsov}
Let $(M,d_M)$ be a metric space with Minkowkski dimension $d \in (0,\infty)$ and finite doubling dimension, and let $(X,\norm{\cdot}_X)$ be a Banach space.
Let $p \in (d,\infty)\cap[1,\infty)$, $\alpha \in (\frac{d}{p},1)$, $\beta \in (0,\alpha - \frac{d}{p})$ and let 
$Z \colon M \times \Omega \to X$ be strongly measurable with $Z \in C^{\alpha}(M,L^p(\Omega;X))$.
Then there exists a continuous modification of $Z$ (again denoted by $Z$) and a constant $C_M\in (0,\infty)$ depending only on the metric space $M$ such that
\begin{align}\label{eq:KC}
    \norm[\big]{\abs{Z}_{C^{\beta}(M,X)}}_{L^p(\Omega)} \leq
    24\, C_M^{\alpha}
    \beta^{-1}
    \left( \tfrac{\alpha - \beta}{\alpha-d/p-\beta}\right)^{1/p} 
    \frac{\Delta(M)^{\alpha}}{\Delta(M)^{\beta}}
    |Z|_{C^{\alpha}(M,L^p(\Omega;X))}.
\end{align}
\end{theorem}

\begin{proof}By rescaling the metric (see \cref{rem:genholder_rescale,cor:dimensions_rescale}), we can assume without loss of generality that $\Delta(M) = 1$. 
Let $w_{\alpha}$, $w_{\beta}$ be as in \cref{example:weight}-\eqref{example:weight:holder}, and recall from \cref{example:Holder} that in this case $\abs{\cdot}_{C^{\alpha}}$ and $\abs{\cdot}_{C_{w_\alpha}}$ are identical (and likewise for $\beta$).
Next, let $(x_k,y_k)_{k \in \N}$ be a sequence obtained from \cref{thm:embedding}, and set $M_0 = \cup_{k \in \N} \cur{x_k,y_k}$.
Then by \cref{thm:embedding,rem:embeddingnoncont} we have
\begin{equation}
    \label{eq:KCstep1}
    \begin{aligned}
         |Z(\cdot,\omega)|_{C_{w_{\beta}}(M_0,X)}
        &\leq 3 \cdot 2^{2\beta}(2^{\beta}-1)^{-1}
 \sup_{k\in \N } \frac{ \| Z(x_k,\omega) - Z(y_k,\omega) \|_X }{w_{\beta}(k^{-1/d})}
 \\ & \leq 24 \beta^{-1}
 \sup_{k\in \N } k^{\beta/d} \| Z(x_k,\omega) - Z(y_k,\omega) \|_X
    \end{aligned}
\end{equation}
for all $\omega \in \Omega$,
recalling that the growth\ constants of $w_{\beta}$ both equal $2^{\beta}$ and using that $2^{\beta}-1\geq \beta \log(2) \geq \beta/2$. Similarly, we obtain that there exists a constant $C_M$ depending only on the metric space $M$ such that
\begin{equation}
\begin{aligned}\label{eq:KCstep2} 
  \sup_{k\in \N } k^{\alpha/d} \| Z(x_k,\cdot) - Z(y_k,\cdot) \|_{L^p(\Omega;X)}
  & = \sup_{k\in \N } \frac{ \| Z(x_k,\cdot) - Z(y_k,\cdot) \|_{L^p(\Omega;X)} }{w_{\alpha}(k^{-1/d})}
 \\ &\leq 
 C_M^{\alpha} \abs{Z}_{C_{w_\alpha}(M,L^p(\Omega))}.
\end{aligned}
\end{equation}
Combining~\eqref{eq:KCstep1},~\eqref{eq:KCstep2}, and \cref{lem:KCLplemma} we obtain 
\begin{equation}\label{eq:KCstep3}
\norm[\big]{\abs{Z}_{C_{w_\beta}(M_0,X)}}_{L^p(\Omega)}
\leq 24\, C_M^{\alpha}\beta^{-1} \left( \tfrac{\alpha - \beta}{\alpha-d/p-\beta}\right)^{1/p}|Z|_{C_{w_\alpha}(M,L^p(\Omega))}.
\end{equation}
Next, define $\tilde{\Omega} = \cur{\omega \in \Omega : |Z(\cdot,\omega)|_{C_{w_{\beta}}(M_0,X)} < \infty}$.
By~\eqref{eq:KCstep3} we see $\PP{\tilde{\Omega}}=1$, and from \cref{rem:genholder_cont} it follows that $Z(\cdot,\omega) \in C(M_0,X)$ for $\omega \in \tilde{\Omega}$.
Since $M_0$ is dense in $(M,d_M)$, this allows us to define $\tilde{Z}\colon M \times \Omega \rightarrow X$ by
\begin{equation} \label{eq:def_contmod}
 \tilde{Z}(z,\omega)
 =
 \begin{cases} 
    \lim_{z_n\in M_0, z_n\rightarrow z }Z(z_n,\omega) ,
    & 
    \omega \in \tilde{\Omega};\\
    0 
    &
    \omega \in \Omega \setminus \tilde{\Omega}
 \end{cases}, \quad z\in M.
\end{equation} 
Note that $\tilde{Z}$ has continuous paths by construction, and that $\tilde{Z}$ is a modification of $Z$ by Fatou's lemma.
Finally,~\eqref{eq:KC} follows from~\eqref{eq:KCstep3} and the fact that $\tilde{Z}$ has continuous sample paths.
\end{proof}

\begin{remark}\label{rem:compareKU_KC}
Theorem 1.1 in~\cite{KratschmerUrusov:2022} considers a slightly more general setting than~\cref{thm:Kolmogorov-Chentsov}. Firstly, the authors assume $Z$ takes values in a general metric space instead of a Banach space -- however, one can easily reduce to the Banach space setting by Kuratowski's embedding. 
More importantly, in~\cite[Theorem 1.1]{KratschmerUrusov:2022}, the metric space $M$ is \emph{not} required to have a finite doubling dimension, so only the assumption on the Minkowski dimension is present.
We expect that~\cite[Lemma 6.1]{KratschmerUrusov:2022} (which is a modification of~\cite[Lemma B.2.7]{Talagrand:2014}) could be used to obtain a variation of~\cref{thm:embedding} (and thus also~\cref{thm:Kolmogorov-Chentsov}) in this setting, although this would require a significantly more technical argument.
This, together with the fact that in our cases of interest it is harmless to assume a finite doubling dimension, deterred us from pursuing this further.
\end{remark}

\section{H\"older regularity for parameter-dependent stochastic integrals}
In this section we combine \cref{thm:seidlerplus} and \cref{thm:embedding} to establish H\"older regularity for parameter-dependent stochastic integrals, see \cref{thm:seidlerplus_Hoelder} below.

Throughout the section, we let $T>0$, let $(X,\norm{\cdot}_X)$ be a $(2,D)$-smooth Banach space and let $W$ be an $H$-cylindrical Wiener processes on $(\Omega, \calF,\P, (\calF_t)_{t \geq 0})$ (recall the stochastic setup from \cref{sec:prelim}).
In addition, we let $(M,d_M)$ be a metric space with Minkowski dimension $d\in (0,\infty)$ and finite doubling dimension (see \cref{def:Minkowski_dim,def:doubling_dim}). 
Finally, let $w\colon (0,1]\rightarrow (0,\infty)$ be an admissible modulus of continuity (see \cref{def:admissible_weight}) and define 
\begin{equation} \label{eq:wlogpd}
w_{\log,p,d}(x)= (p-d\log(x))^{-1/2} w(x), \quad x\in (0,1],\, p \in [1,\infty).
\end{equation}
Note that by \cref{example:weight} \eqref{example:weight:rescale}-\eqref{example:weight:modinvlog},
$w_{\log,p,d}$ is again an admissible modulus of continuity, with growth constants which can be bounded independently of $p$.

\begin{theorem}\label{thm:seidlerplus_Hoelder}
Let $p\in [1,\infty)$, and let $((\Phi(x)(t))_{t \in [0,T]})_{x \in M}$ be a family of progressive $\gamma(H,X)$-valued processes indexed by $M$, which additionally satisfies
\begin{equation}
    K(\Phi) \coloneq \norm[\big]{\abs{\Phi}_{C_{w_{\log,p,d}}( M , L^2(0,T; \gamma(H,X)))}}_{L^p(\Omega)} < \infty.
\end{equation}
Define $I(\Phi)\colon M \rightarrow L^p(\Omega;X)$ by 
\begin{equation}
 I(\Phi)(x) = \int_0^{T} \Phi(x)(t)\rd W(t), \quad x \in M.
\end{equation}
Then there exists a continuous modification of $I(\Phi)$ (again denoted by $I(\Phi)$), and we have the estimate
\begin{equation}\label{eq:seidlerplus_Hoelder}
\norm[\big]{ \abs{I(\Phi)}_{C_w(M,X)} }_{L^{p}(\Omega)}
\leq D\, C_{\eqref{eq:seidlerplus_Hoelder}} K(\Phi),
\end{equation}
for some $C_{\eqref{eq:seidlerplus_Hoelder}} > 0$ which depends only on $(M,d_M)$ and $w$.
\end{theorem}
\begin{remark}
    The right-hand side of \eqref{eq:seidlerplus_Hoelder} still depends on $p$ via  $w_{\log,p,d}$.
    From \cref{def:generalized_holder} and \eqref{eq:wlogpd}, it can be seen that the right-hand side of \eqref{eq:seidlerplus_Hoelder} blows up at a rate $\calO(\sqrt{p})$ as $p \to \infty$, meaning that we retain the correct scaling in $p$.
\end{remark}

From~\eqref{eq:seidlerplus_Hoelder} it is now clear how much regularity is needed in order to have a stochastic integral which is H\"older continuous.
This is demonstrated by using $w\colon (0,1] \to (0,\infty)$ given by  $w(x) = \sqrt{p - d \log(x)}\, x^{\alpha}$ for some $\alpha \in (0,1)$ (see \cref{example:weight} \eqref{example:weight:holderlog} 
and \cref{rem:logmodrestriction}).
In this case, it follows from \cref{prop:holder_logblowup} that if $\Phi$ has regularity $C^{\alpha}$, then $I(\Phi)$ has regularity $C^{\alpha - \eps}$ for every $\eps \in (0,\alpha)$, and the $C^{\alpha-\eps}$-norm of $I(\Phi)$ blows up (pointwise) at a rate $\calO(\eps^{-1/2})$ as $\eps \to 0$.
We remark that this rate of blowup (modulo constants) does not depend the properties of the underlying metric space.
\begin{proof}[Proof of \cref{thm:seidlerplus_Hoelder}]
By rescaling the metric (see \cref{rem:genholder_rescale,cor:dimensions_rescale}), we can assume without loss of generality that $\Delta(M) = 1$.
We let $(x_k,y_k)_{k \in \N}$ be a sequence obtained from \cref{thm:embedding}, and set $M_0 = \cup_{k \in \N} \cur{x_k,y_k}$. 
By \eqref{eq:metric_Ciesielski_seminorm} and \cref{rem:embeddingnoncont}, we then obtain the inequalities
\begin{subequations}
\begin{align} \label{eq:stochint_embedding1}
    \abs{I(\Phi)}_{C_w(M_0,X)}
 &\leq C_{\eqref{eq:embeddingconstants2}} 
 \sup_{k\in \N } \frac{ \| I(\Phi)(x_k) - I(\Phi)(y_k) \|_X }{w(k^{-1/d})}, \\
 \label{eq:stochint_embedding2}
  \sup_{k\in \N } \frac{ \| \Phi(x_k) - \Phi(y_k) \|_{L^2(0,T;\gamma(H,X))} }{w_{\log,p,d}(k^{-1/d})}
 &\leq 
 C_{\eqref{eq:embeddingconstants1}} \abs{\Phi}_{C_{w_{\log,p,d}}(M,L^2(0,T;\gamma(H,X)))},
\end{align}
where the doubling constants of $w$ (resp. $w_{\log,p,d}$) should be used to compute $C_{\eqref{eq:embeddingconstants2}}$ (resp. $C_{\eqref{eq:embeddingconstants1}}$).
After observing the identity
\begin{equation*}
    \frac{\sqrt{p + \log(k)}}{w(k^{-1/d})}
    = \frac{\sqrt{p - d\log(k^{-1/d})}}{w(k^{-1/d})}
    = \frac{1}{w_{\log,p,d}(k^{-1/d})},\quad k \in \N,
\end{equation*}
it follows from a direct application of \cref{thm:seidlerplus} that
\begin{equation}
\label{eq:seidlerplus_preHoelder1}
\begin{aligned}
 &\left\| 
    \sup_{k\in \N } \frac{ \| I(\Phi)(x_k) - I(\Phi)(y_k) \|_X }{w(k^{-1/d})}
 \right\|_{L^p(\Omega)}
 \\ &\quad \leq 
 10 D
 \left\| 
    \sup_{k\in \N } \frac{ \| \Phi(x_k) - \Phi(y_k) \|_{L^2(0,T;\gamma(H,X))} }{ w_{\log,p,d}(k^{-1/d})}
 \right\|_{L^p(\Omega)}.
\end{aligned}
\end{equation}
\end{subequations}
Combining~\eqref{eq:stochint_embedding1},~\eqref{eq:stochint_embedding2}, and~\eqref{eq:seidlerplus_preHoelder1}, we obtain
\begin{equation}\begin{aligned}\label{eq:seidlerplus_preHoelder2}
\norm[\big]{\abs{I(\Phi)}_{C_w(M_0,X)}}_{L^p(\Omega)}
  \leq 
 10 D C_{\eqref{eq:embeddingconstants2}} C_{\eqref{eq:embeddingconstants1}} 
 K(\Phi).
 \end{aligned}
\end{equation}
Next, we define $\tilde{\Omega} = \cur{\omega \in \Omega : \abs{I(\Phi)}_{C_w(M_0,X)} < \infty}$  and define the modification of $I(\Phi)$ analogously to~\eqref{eq:def_contmod}.
The remainder of the argument is entirely analogous to the one provided in the proof of \cref{thm:Kolmogorov-Chentsov}, noting that if $z\in M$ and $(z_k)_{k\in \N}$ is a sequence in $M_0$ converging to $z$, then Fatou's lemma and~\cite[Theorem 1.1]{Seidler:2010} imply 
\begin{equation*}
\begin{aligned}
 & \| \lim_{k\rightarrow \infty} \| I(\Phi)(z) - I(\Phi)(z_k)\|_X \|_{L^p(\Omega)}
  \leq 
 \lim_{k\rightarrow \infty}  \| I(\Phi)(z) - I(\Phi)(z_k)\|_{L^p(\Omega;X)}
 \\ & \qquad  
 \leq 10 D \sqrt{p} \lim_{k\rightarrow \infty} \| \Phi(z_k) - \Phi(z)\|_{L^p(\Omega;L^2(0,T;\gamma(H,X)))}
 =0
 \end{aligned}
\end{equation*}
because $K(\Phi) < \infty$.
\end{proof}
\begin{remark}\label{rem:LevyModCont}
If we take $T=1$, $H=X=\R$ (so $\gamma(H,X)\simeq \R$), $M=[0,1]$ endowed with the Euclidean metric, $w(x)=\sqrt{1-\tfrac{1}{2}\log(x))x}$, and $\Phi(s)(t) = 1_{[0,s]}(t)$ ($s,t\in [0,1]$) in the setting of \cref{thm:seidlerplus_Hoelder}, then $I(\Phi)(s)=W(s)$ and \cref{thm:seidlerplus_Hoelder} implies
\begin{equation}\label{eq:seidlerplus_Hoelder_simple}
 \left\| 
    \sup_{0\leq r< s \leq 1}
    \frac{|W(s)-W(r)|}{\sqrt{(s-r) (1-\frac{1}{2}\log(s-r))}}
 \right\|_{L^p(\Omega)}
 \leq 
 C \sqrt{p}
\end{equation}
for all $p\in [1,\infty)$ (where $C\in (0,\infty)$ is independent of $p$). Recall that
L\'evy's modulus of continuity theorem states that
\begin{equation}\label{eq:Levymodulus}
\lim_{h\downarrow 0} 
    \sup_{0\leq r< s \leq 1,\, |r-s|< h}
    \frac{|W(s)-W(r)|}{\sqrt{2h |\log(h)|}}
    = 
    1 \quad \text{a.s.}
\end{equation} 
Comparing~\eqref{eq:seidlerplus_Hoelder_simple} and~\eqref{eq:Levymodulus} we see that \cref{thm:seidlerplus_Hoelder} is sharp in terms of the obtained modulus of continuity.
\end{remark}

\section{Application: regularity of the 1D parabolic Anderson model}
\label{sec:regPAM}
As an application of \cref{thm:seidlerplus_Hoelder}, we investigate the regularity of the 1D parabolic Anderson model (PAM), which can also be viewed as a stochastic heat equation with linear multiplicative noise.
To formulate the equation, we fix $\eta \in (0,\infty)$, $T \in (0,\infty)$ and let $W$ be an $L^2(0,1)$-cylindrical Wiener process on $(\Omega,\calF,\P,(\calF_t)_{t\geq 0})$ (recall the stochastic setup from \cref{sec:prelim}).
The parabolic Anderson model is then formally given by:
\begin{subequations}\label{eq:PAM}
\begin{align}\label{eq:PAM_DE}
\tfrac{\partial}{\partial t}U(t,x) &= \tfrac{\partial^2}{\partial x^2} U(t,x) + \eta U(t,x)\,
\tfrac{\partial}{\partial t} W(t,x), & & (t,x) \in (0,T) \times (0,1), \\
\label{eq:PAM_BC}
U(t,0)&=U(t,1)  = 0, & &  t \in (0,T], \\
\label{eq:PAM_IC}
U(0,x)& = U_0(x), & & x \in [0,1],
\end{align}
\end{subequations}
where the (random) initial condition $U_0$ is assumed to be $\calF_0$-measurable.
To avoid complicating the presentation, we assume $U_0 \in L^p(\Omega;C^2([0,T]))$ for some $p\in (4,\infty)$, so that the regularity of $U$ will not be limited by that of the initial value.

To obtain a solution to \eqref{eq:PAM}, we let $(h_k)_{k\in \N}$ be an orthonormal basis of $L^2(0,1)$ and set $\beta_k(t) = W(h_k \otimes 1_{[0,t]})$ ($k\in \N$, $t\geq 0$), rendering $(\beta_k)_{k\in\N}$ a sequence of independent standard Brownian motions. 
We also let $G\colon [0,\infty)\times [0,1] \times [0,1] \rightarrow \R$ be the Green's function associated with the Dirichlet Laplacian on $[0,1]$, i.e., 
\begin{equation}\label{eq:Green} 
G(t,x,y) = \sum_{k\in \N} 2 e^{-\pi^2 k^2 t}\sin(\pi k x)\sin(\pi k y).
\end{equation}
It then follows from~\cite[Theorem 6.2]{vanNeervenVeraarWeis:2008} that there exists a unique (up to indistinguishability) adapted stochastic process $U \in L^p(\Omega; C([0,T]\times [0,1]))$ such that
\begin{equation}\label{eq:solPAM}\begin{aligned}
U(t,x) & =\int_0^1 G(t,x,y) U_0(y)\rd y \\
& \quad + \eta \sum_{k\in \N} \int_0^t \int_0^1 G(t-s,x,y)U(s,y)h_k(y) \rd y \rd\beta_k(s).
\end{aligned}
\end{equation}
The process $U$ is conventionally called the \emph{mild solution} to \eqref{eq:PAM}.
From the Sobolev embedding theorem and \cite[Theorem 6.3]{vanNeervenVeraarWeis:2008},
it also follows that $U \in L^p(\Omega;C^{\lambda}([0,T],C^{2\gamma}([0,1])))$ for every $\lambda,\gamma \in (0,\infty)$ satisfying $\lambda + \gamma < \frac{1}{4}-\frac{1}{p}$.
In particular, we have
\begin{equation}
\left\| \sup_{\substack{t,s\in [0,T]; x,y\in [0,1];\\ (t,x)\neq (s,y)}} \frac{|U(t,x)-U(s,y)|}{|t-s|^{\frac{1}{4}-\frac{1}{p}-\eps} + |x-y|^{\frac{1}{2}-\frac{2}{p}-\eps}} \right\|_{L^p(\Omega)} < \infty,
\end{equation}
for all sufficiently small $\eps > 0$. 
We now demonstrate how \cref{thm:seidlerplus_Hoelder} can be employed to show the following refined regularity result:

\begin{theorem}\label{thm:regPAM}
Let $p \in (4,\infty)$, $U_0 \in L^p(\Omega;C^2([0,1]))$, and let $U$ be the unique mild solution to \eqref{eq:PAM}.
Then we have
\begin{equation}\label{eq:regPAM2}
\norm[\Bigg]{ \sup_{\substack{t,s\in [0,T];\\ x,y\in [0,1];\\ (t,x)\neq (s,y)}} \frac{|U(t,x)-U(s,y)|}{(1-\tfrac{1}{4}\log(|t-s|))^{\frac{1}{2}} |t-s|^{\frac{1}{4}} + (1-\tfrac{1}{2}\log(|x-y|))  |x-y|^{\frac{1}{2}}}}_{L^p(\Omega)} < \infty.
\end{equation}
\end{theorem}

\begin{proof}
Throughout the proof, we write $X \lesssim Y$ if there exists a constant $C$ depending only on $p$ and $T$, such that $X \leq C Y$.

Note that the first term on the right-hand side of \eqref{eq:solPAM} solves the heat equation with initial value $U_0$.
Thus, by the assumed regularity of $U_0$, we only need to concern ourselves with the stochastic integral in \eqref{eq:solPAM}.
We begin by noting that every $h\in L^2(0,1)$ can be associated with an operator $R_h\in \gamma(L^2(0,1),\R)$ via the relation $R_h g = \int_{0}^1 h(x)g(x)\rd x$,
in which case $\norm{R_h}_{\gamma(L^2(0,1),\R)} = \norm{h}_{L^2(0,1)}$ (see~\cite[Proposition 9.2.9]{HytonenEtAl:2017}). 
In particular, setting $M=[0,T]\times [0,1]$, this induces an isometric isomorphism
\begin{equation}
    C(M,L^2(M)) \simeq C(M,L^2(0,T;\gamma(L^2(0,1),\R))),
\end{equation}
so that we may define $\Phi \in L^p(\Omega; C(M,L^2(0,T;\gamma(L^2(0,1),\R))))$ by setting
\begin{equation*}
\Phi(\omega)(t,x)(s,y)
= G(t-s,x,y)U(s,y,\omega)1_{[0,t]}(s),
\end{equation*}
whence we have, with $I(\Phi)$ as in \cref{thm:seidlerplus_Hoelder}:
\begin{equation*}
    I(\Phi)(t,x) =\sum_{k\in \N} \int_0^t \int_0^1 G(t-s,x,y)U(s,y)h_k(y) \rd y \rd \beta_k(s), \quad (t,x)\in [0,T]\times [0,1].
\end{equation*}

Our intermediate goal is now to rewrite \eqref{eq:regPAM2} in such a way that we can apply \cref{thm:seidlerplus_Hoelder}.
To this end, we endow $M$ with the following metric:
\begin{equation}
    \label{eq:pammetric}
    \tilde{d}((t,x),(s,y)) = \abs{t - s}^{1/2} + (1 - \tfrac{1}{2}\log (\abs{x - y}))\abs{x-y}.
\end{equation}
Note that $\tilde{d}$ satisfies the triangle inequality, since $t \mapsto t^{1/2}$ is monotone and subadditive on $[0,T]$, and likewise for $x \mapsto (1-\tfrac{1}{2}\log(x))x$ on $[0,1]$.
One can verify that $(M,\tilde{d})$ has finite doubling dimension and Minkowski dimension $3+\eps$ for any $\eps > 0$ (we will only use $\eps=1$).
Also, for any $u,v \in M$ with $u = (t,x)$ and $v = (s,y)$ we have
\begin{align*}
    (1 - \tfrac{1}{2}\log(\tilde{d}(u,v)))\tilde{d}(u,v) 
    &= (1 - \tfrac{1}{2}\log(\tilde{d}(u,v))) \abs{t-s}^{1/2} \\
    &\quad+ (1 - \tfrac{1}{2}\log(\tilde{d}(u,v)))(1 - \tfrac{1}{2}\log(\abs{x-y}))\abs{x-y} \\
    &\leq (1 - \tfrac{1}{4}\log(\abs{t - s}))\abs{t-s}^{1/2} \\
    &\quad+ (1 - \tfrac{1}{2}\log(\abs{x-y}))^2 \abs{x-y},
\end{align*}
since $\tilde{d}(u,v) \geq \max(\abs{t-s}^{1/2},\abs{x-y})$.
Hence, to show \eqref{eq:regPAM2} it suffices to establish
\begin{equation*}
    K \coloneq \norm[\bigg]{\sup_{u,v \in M} \frac{I(\Phi)(u) - I(\Phi)(v)}{(1 - \tfrac{1}{2}\log(\tilde{d}(u,v)))^{1/2} \tilde{d}(u,v)^{1/2}}}_{L^p(\Omega)} < \infty.
\end{equation*}
We now set $w(x) = (1 - \tfrac{1}{2}\log(x)))^{1/2}x^{1/2}$ and note that $w$ is an admissible modulus of continuity by \cref{example:weight}-\eqref{example:weight:holderlog}.
Recalling \cref{def:generalized_holder,example:Holder} (both of which should be interpreted with respect to $\tilde{d}$), we apply \cref{thm:seidlerplus_Hoelder} to see that
\begin{equation}\label{eq:boundK}
\begin{aligned}
    K &\lesssim 
    \norm[\big]{\abs{I(\Phi)}_{C_w(M,\R)}}_{L^p(\Omega)}
    \lesssim
    \norm[\big]{\abs{\Phi}_{C_{w_{\log,p,4}}(M,L^2(0,T;\gamma(L^2(0,1),\R)))}}_{L^p(\Omega)} \\
    &= \norm[\big]{\abs{\Phi}_{C_{w_{\log,p,4}}(M,L^2(M))}}_{L^p(\Omega)}
    \lesssim \norm[\big]{\abs{\Phi}_{C^{1/2}(M,L^2(M))}}_{L^p(\Omega)},
\end{aligned}\end{equation}
where $w_{\log,p,4}$ is as in \eqref{eq:wlogpd} and we have used $w_{\log,p,4}(x) \gtrsim w_{\log,1,1/2}(x) = x^{1/2}$ for the final step (see \cref{rem:logmodrestriction}) --- also note that H\"older regularity in the space $C^{1/2}(M,L^2(M))$ is measured with respect to the metric $\tilde{d}$, whereas $L^2(M)$ is simply the usual Lebesgue space (i.e., involving the Lebesgue measure on $M$).
In view of~\eqref{eq:boundK}, all that remains is to show that
$\norm[\big]{\abs{\Phi}_{C^{1/2}(M,L^2(M))}}_{L^p(\Omega)} < \infty$.
To see this, we observe that for any $t,s \in [0,T]$ and $x,y \in [0,1]$ we have
\begin{align*}
\begin{aligned}
 \norm{& \Phi(t,x) - \Phi(s,y) }_{L^2([0,T] \times [0,1])}^2
\leq \sup_{(\tau,\xi)\in [0,T]\times [0,1]}|U(\tau,\xi)|^2
\\ & \qquad \times 
\int_{0}^{t} \int_0^{1} |G(t-r,x,z) - G(s-r,y,z)1_{[0,s]}(r)|^2 \rd z \rd r.
\end{aligned}
\end{align*}
Applying \cref{lem:regHeatKernel} and taking square roots, we thus find that for any $u,v \in M$, we have
\begin{equation}
\label{eq:PhiHoelderbound}
    \norm{\Phi(u) - \Phi(v)}_{L^2(M)} \lesssim \tilde{d}(u,v)^{1/2}\norm{U}_{C(M,\R)},
\end{equation}
so that $\norm{\Phi}_{C^{1/2}(M,L^2(M))} \lesssim \norm{U}_{C(M,\R)}$ by \cref{example:Holder}.
We conclude by taking the $L^p(\Omega)$-norm and using $\norm{U}_{L^p(\Omega;C(M,\R))} < \infty$.
\end{proof}

\begin{remark}\label{rem:generalIC}
We have assumed $U_0 \in L^p(\Omega;C^2([0,1]))$ to avoid any regularity issues coming from the initial value.
Similarly, the assumption $p > 4$ is only used to guarantee existence of a solution in $L^p(\Omega;C([0,T] \times [0,1]))$ from~\cite[Theorem 6.2]{vanNeervenVeraarWeis:2008}, see~\eqref{eq:PhiHoelderbound}.
We expect that with some additional bookkeeping (which would distract from our presentation) \cref{thm:regPAM} can straightforwardly be extended to rougher initial values, such as $U_0 \in L^p(\Omega;C([0,1]))$ or beyond.
One could even forego the use of \cite{vanNeervenVeraarWeis:2008} entirely by performing a fixed-point argument in the space $L^p(\Omega;C([0,T]\times [0,1]))$ with $p\in [1,\infty)$ using the estimates outlined in the proof of \cref{thm:regPAM}.
\end{remark}

\appendix 
\section{Regularity of the Green's function}
The proof of \cref{thm:regPAM} relies on the following regularity result of the Dirichlet Green's function.

\begin{lemma}\label{lem:regHeatKernel}
Let $G\colon [0,\infty) \times [0,1]\times[0,1] \to \R$ be the Dirichlet heat kernel given by~\eqref{eq:Green} and let $T>0$. 
Then there exists a constant $c_{\eqref{eq:regHeatKernel}}\in (0,\infty)$ (possibly depending on $T$) such that
\begin{equation}\label{eq:regHeatKernel}
\begin{aligned}
& \int_{0}^{t} \int_0^{1} \abs{G(t-r,x,z) - G(s-r,y,z)1_{[0,s]}(r)}^2\rd z \rd r
\\ & \qquad \leq c_{\eqref{eq:regHeatKernel}}\bra[\big]{ \abs{t-s}^{1/2} - \log(\abs{x-y})\abs{x-y}}
\end{aligned}
\end{equation}
for all $t,s\in [0,T]$ and $x,y\in [0,1]$ satisfying $s<t$.
\end{lemma}

\begin{proof}[Proof of \cref{lem:regHeatKernel}]
Fix $x,y \in [0,1]$ and $0 \leq s < t \leq T$.
Throughout the proof, we write $A \lesssim B$ if there exists a constant $C$, independent of $x,y,t,s$ such that $A \leq C B$.

First, using \eqref{eq:Green} and Parseval's identity (in the $z$ variable) we can estimate
\begin{align*}
    &\int_s^t \int_0^1 \abs{G(t-r,x,z)}^2 \rd z \rd r
    \lesssim \int_s^t \sum_{k \in \N}e^{-2 \pi^2 k^2 (t-r)} \rd r \\
    &\qquad = \int_0^{t-s}\sum_{k \in \N}e^{-2\pi^2k^2 r} \rd r
    \leq \int_0^{t-s}r^{-1/2} \int_0^{\infty} e^{-2\pi^2 z^2}\rd z \rd r \lesssim 2(t-s)^{1/2}.
\end{align*}

By additionally using Minkowski's inequality, we also find
\begin{align*}
    &\int_0^s \int_0^1 \abs{G(t-r,x,z)-G(s-r,x,z)}^2 \rd z \rd r \\
    &\qquad \lesssim \int_0^s \sum_{k\in\N} \bra[\big]{e^{-\pi^2 k^2 (s-r)} - e^{-\pi^2k^2 (t-r)}}^2 \rd r \\
    &\qquad\lesssim \int_0^s \sum_{k \in \N}\bra[\bigg]{\int_{s-r}^{t-r} k^2e^{-\pi^2 k^2 \tau} \rd \tau}^2 \rd r \\
    &\qquad \lesssim \int_0^s \bra[\Big]{\int_{s-r}^{t-r} \bra[\Big]{\sum_{k \in \N} k^4 e^{-2\pi^2 k^2 \tau}}^{1/2}\rd \tau}^2 \rd r \\
    &\qquad \lesssim \int_0^s \bra[\Big]{\int_{s-r}^{t-r} \tau^{-5/4}\rd \tau}^2 \rd r \\
    &\qquad  \lesssim \int_0^{s} \bra[\Big]{
  (s-r)^{-1/4} - (t-r)^{-1/4}
   }^2 \rd r
\end{align*}
Using the fact that $(b-a)^2 \leq b^2 - a^2$ whenever $0\leq a \leq b$, we obtain 
\begin{align*}\int_0^s \bra[\Big]{
   & (s-r)^{-1/4} - (t-r)^{-1/4}
   }^2 \rd r
  \leq 
  \int_0^{s} 
  (s-r)^{-1/2} - (t-r)^{-1/2}
  \rd r
  \\ & \qquad
  = 
  2\bra[\big]{s^{1/2} - t^{1/2} + (t-s)^{1/2}} 
  \leq 
  2(t-s)^{1/2}.
\end{align*}

Finally, using \eqref{eq:Green}, Parseval's identity, and the H\"older continuity of the sine function, we obtain for $\eps \in (0,1)$:
\begin{align*}
    &\int_0^s \int_0^1 \abs{G(s-r,x,z) - G(s-r,y,z)}^2 \rd z \rd r \\
    &\qquad= \int_0^s \int_0^1 \abs{G(r,x,z) - G(r,y,z)}^2 \rd z \rd r \\
    &\qquad \lesssim \int_0^s \sum_{k \in \N}e^{-2\pi^2 k^2 r}(\sin(k\pi x) - \sin(k \pi y))^2\rd r \\
    &\qquad \lesssim \sum_{k \in \N}k^{-2}(\sin(k\pi x) - \sin(k \pi y))^2 \\
    &\qquad \lesssim \sum_{k \in \N}k^{-(1+\eps)} \abs{x-y}^{1-\eps} \\
    &\qquad \lesssim \eps^{-1}\abs{x-y}^{1-\eps}.
\end{align*}
Choosing $\eps^{-1} = -\log \abs{x-y}$ and combining with the previous estimates, the proof is complete.
\end{proof}

\bibliographystyle{amsalpha}
\bibliography{coxwinden}
\end{document}